%% file: rmn-2017.tex
\documentclass[11pt]{article}

\usepackage{macros}
\usepackage[utf8]{inputenc}
\usepackage{amsmath,amssymb,amsthm}
\usepackage{multirow}
\usepackage{booktabs} 
\usepackage{lineno}
\usepackage{cite}
\usepackage[position=b]{subcaption}
\usepackage[margin=1.0in]{geometry}
\usepackage[textsize=footnotesize,textwidth=3.5cm]{todonotes}

\newcommand*\patchAmsMathEnvironmentForLineno[1]{%
  \expandafter\let\csname old#1\expandafter\endcsname\csname #1\endcsname
  \expandafter\let\csname oldend#1\expandafter\endcsname\csname end#1\endcsname
  \renewenvironment{#1}%
     {\linenomath\csname old#1\endcsname}%
     {\csname oldend#1\endcsname\endlinenomath}}%
\newcommand*\patchBothAmsMathEnvironmentsForLineno[1]{%
  \patchAmsMathEnvironmentForLineno{#1}%
  \patchAmsMathEnvironmentForLineno{#1*}}%
\AtBeginDocument{%
\patchBothAmsMathEnvironmentsForLineno{equation}%
\patchBothAmsMathEnvironmentsForLineno{align}%
\patchBothAmsMathEnvironmentsForLineno{flalign}%
\patchBothAmsMathEnvironmentsForLineno{alignat}%
\patchBothAmsMathEnvironmentsForLineno{gather}%
\patchBothAmsMathEnvironmentsForLineno{multline}%
}

\def\etad{\eta_*}
\def\tf{t_{\rm f}}

\def\vhat{\hat \bv}
\def\vhata{\vhat(\bsalpha)}
\def\vhatA{\vhat(\cA)}

\def\vtilde{\widetilde{\bv}}

\def\etad{\eta_{\ast}}
\def\hd{h_{\ast}}
\def\jd{j_{\ast}}

\def\hg{h_\gamma}
\def\hgv{\hg(\bv)}
\def\hgd{(h_\gamma)_\ast}

\def\jg{j_\gamma}

\def\vhatg{\vhat_\gamma}
\def\vhatga{\vhatg(\bsalpha)}
\def\vhatgA{\vhatg(\cA)}

\def\alphahat{\hat{\bsalpha}}
\def\alphahatg{\alphahat_\gamma}
\def\alphahatu{\alphahat(\bu)}
\def\alphahatv{\alphahat(\bv)}
\def\alphahatgu{\alphahatg(\bu)}
\def\alphahatgv{\alphahatg(\bv)}

\def\alphabar{\overline{\bsalpha}}

\def\Ga{G_{\bsalpha}}
\def\Gu{G_{\alphahatu}}
\def\Gv{G_{\alphahatv}}
\def\Ggu{G_{\alphahatgu}}
\def\Ggv{G_{\alphahatgv}}

\begin{document}

\title{A regularized entropy-based moment method for kinetic equations%
\footnote{
This manuscript has been authored by UT-Battelle, LLC under Contract No. 
DE-AC05-00OR22725 with the U.S. Department of Energy. The United States 
Government retains and the publisher, by accepting the article for publication, 
acknowledges that the United States Government retains a non-exclusive, paid-up, 
irrevocable, world-wide license to publish or reproduce the published form of 
this manuscript, or allow others to do so, for United States Government 
purposes. The Department of Energy will provide public access to these results 
of federally sponsored research in accordance with the DOE Public Access Plan 
(\texttt{http://energy.gov/downloads/doe-public-access-plan}).
}
}

\date{\today}
\author{Graham W. Alldredge 
\thanks{Department of Mathematics and Computer Science,
Freie Universit\"at Berlin,
14195 Berlin, Germany,
(\texttt{graham.alldredge@fu-berlin.de})
}
\and
Martin Frank%
\thanks{Department of Mathematics,
Karlsruhe Institute of Technology,
D-76128 Karlsruhe, Germany,
(\texttt{martin.frank@kit.edu}).}
\and
Cory D. Hauck%
\thanks{Computational Mathematics Group,
Computer Science and Mathematics Division,
Oak Ridge National Laboratory,
Oak Ridge, TN 37831, USA, 
(\texttt{hauckc@ornl.gov}).}
}

\maketitle

\input{sections/intro}

\input{sections/background}

\input{sections/rmn}

\input{sections/accuracy}

\input{sections/numerics}

\input{sections/conclusions}

\appendix

\input{sections/junk}

\input{sections/ack}

\bibliographystyle{plain}
\bibliography{mnrefs}

\end{document}

%% file: sections/intro.tex
\section{Introduction}

Kinetic equations model systems consisting of a large number of particles that
interact with each other or with a background medium.
They arise in a wide 
variety of applications, including
rarefied gas dynamics \cite{Cercignani},
neutron transport \cite{Lewis-Miller-1984},
radiative transport \cite{mihalas1999foundations},
and semiconductors \cite{markowich1990}.
For charge-neutral particles, these equations evolve the \textit{kinetic density 
function}
$f \colon [0, \infty) \times X \times V \to [0, \infty)$ according to
\begin{equation}\label{eq:kinetic}
 \partial_t f(t, x, v) + v \cdot \nabla_x f(t, x, v) = \cC(f(t, x, \cdot))(v).
\end{equation}
The function $f$ depends on time $t \in [0, \infty)$, position
$x \in X \subseteq \R^d$, and a velocity variable $v \in V \subseteq \R^d$.
The operator $\cC$ introduces the effects of particle collisions; at each $x$ 
and $t$, it 
is an integral operator in $v$.
In order to be well-posed, \eqref{eq:kinetic} must be accompanied by appropriate
initial and boundary conditions.

In this work, we present a new entropy-based moment method for the velocity 
discretization of 
\eqref{eq:kinetic}.   The method relies on a regularization of the optimization 
problem that defines
the closure in the moment equations.
The key advantage of our approach is that, unlike the standard entropy-based 
method,
the solution of the moment equations in the regularized setting is not required 
to take on 
realizable values. 
Roughly speaking, a vector is said to be realizable if it is the velocity 
moment 
of a scalar-valued kinetic density function that takes values in a prescribed 
range.
Typically this range is the set of nonnegative values, but in some cases, an 
upper bound is also enforced.
In practical applications, it is advantageous to remove the requirement of 
realizability because it has 
proven to be difficult to design numerical methods, particularly high-order 
ones, that can maintain it.


Before introducing the regularized method, in \secref{background} we provide 
the necessary background on moment methods, particularly with the 
entropy-based approach.
In \secref{structure}, we introduce the new method and show that it retains 
many, though not all, of the attractive structural properties of the original 
approach.
We then show in \secref{acc} that the new method can be used to generate 
accurate numerical simulations of standard entropy-based moment equations, 
thereby bypassing the need to design a realizable solver for them.
In \secref{numerics}, we demonstrate the accuracy of such simulations using 
the method of manufactured solutions and a benchmark problem.

%% file: sections/background.tex
\section{Background}
\label{sec:background}

In this section, we briefly review the formalism for entropy-based moment
methods.  
The key topics are: structural properties of the kinetic equations, the general
moment approach, the entropy-based closure, and the issue of realizability.
Throughout the discussion and for the remainder of the paper, we rely on 
bracket 
notation for velocity integration: for any $g \in L^1(V)$,
\begin{align}
\Vint{g} := \int_V g(v) \intdv.
\end{align}

\subsection{Structure of the kinetic equation}


The structure of the kinetic equation \eqref{eq:kinetic} plays a definitive 
role in the design of moment methods (and numerical methods in general).
This structure is induced by properties of the collision operator $\cC$ and the 
advection operator $\cA = \partial_t + v \cdot \nabla_x$.
We highlight the basic structural elements below, which are satisfied in many
situations.

\begin{enumerate}[(i)]
	\item \emph{Invariant range}:  There exists a set $B \subseteq [0, \infty)$, 
	consistent with the physical bounds on $f$, such that
	$\range(f(t,\cdot,\cdot)) \subseteq B$ whenever
	$\range(f(0,\cdot,\cdot)) \subseteq B$.
	In general, $f$ is expected to be nonnegative becaues it is a density; for 
	particles satisfying Fermi-Dirac statistics, it should also be bounded from 
	above.

	\item \emph{Conservation}: There exist functions $\phi \colon V \to \R$,
	called \textit{collision invariants}, such that
	\begin{align}\label{eq:invariants}
	\Vint{\phi\, \cC(g)} = 0, \qquad \text{for all } g \in \text{Dom}(\cC).
	\end{align}
	We denote the linear span of all collision invariants by $\bbE$.
%
	When combined with the kinetic equation, \eqref{eq:invariants} 
	implies local conservation laws of the form:
	\begin{align}\label{eq:conservation}
	\partial_t \Vint{\phi f} + \nabla_x \cdot \Vint{v \phi f} = 0.
	\end{align}
	\item \emph{Hyperbolicity}:  For each fixed $v$, the advection operator 
	is hyperbolic over  $(t,x) \in [0, \infty) \times X$.
	\item \emph{Entropy dissipation}:
	Let $D \subseteq \R$.
	There exists a twice continuously differentiable, strictly convex 
	function $\eta \colon D \to \R$, called the \textit{kinetic entropy density}, 
	such that
	\begin{align}\label{eq:entropy-diss}
		\Vint{\eta'(g) \cC(g)} \le 0 \qquad \text{for all }
		g\in \text{Dom}(\cC) \text{ such that } \text{Range}(g) \subseteq D.
	\end{align}
  Combined with the kinetic equation, \eqref{eq:entropy-diss} 
  implies the local entropy dissipation law
  \begin{align}\label{eq:entropy-diss-kin-eq}
   \partial_t \Vint{\eta(f)} + \nabla_x \cdot \Vint{v \eta(f)} \le 0.
  \end{align}
  Often $D$ is consistent with physical bounds on the range of 
  $f$, i.e., $B = D$.
  (See Table \ref{tab:ex-ent} below.)
	\item \emph{H-Theorem}:  Equilibria are characterized by any of the three 
	equivalent statements:
  \begin{align}
		{\rm (a)}~\Vint{\eta'(g) \cC(g)} = 0;
		\qquad
		{\rm (b)}~\cC(g) = 0; 
		\qquad
		{\rm (c)}~\eta'(g) &\in \bbE.
	\end{align}
	\item \emph{Galilean invariance}:
	There exist Galilean transformations $\cG_{O,w}$ defined by
	\begin{align}\label{eq:Galilean}
	{(\cG_{O,w} g)(t, x, v) := g(t, O(x - t w), O(v - w))},
	\end{align}
	where $O \in \operatorname{SO}(d)$ is a $d \times d$ rotation matrix and
	$w \in V$ is a translation in velocity,
	that commute with the advection and collision operators. i.e.,
  \begin{alignat}{3}\label{eq:invariant}
   \cA(\cG_{O,w} g) &= \cG_{O,w} \cA(g) \quad && \text{for all }  g \in 
    \operatorname{Dom}(\cA) \\
   \cC(\cG_{O,w} g) &= \cG_{O,w} \cC(g) \quad && \text{for all }  g \in 
    \operatorname{Dom}(\cC).
  \end{alignat}
  As a consequence, the transformed particle 
  density $\cG_{O,w} f$ also satisfies the kinetic equation \eqref{eq:kinetic}.
 
\end{enumerate}

\begin{table}
	\renewcommand{\arraystretch}{2}
	\centering
	\small
	\begin{tabular}{c|c|c|c|c|c}
		Entropy type 			& $\eta(z)$ 							& $\dom(\eta)$ 		& $\eta'(z)$ 	
& $\etad(y)$ 	& $\etad'(y)$ \\ \hline
		Maxwell--Boltzmann 	& $z \log(z) - z$	 					& $[0, \infty)$ 	& 
$\log(z)$		& $e^y$			& $e^y$ \\
		{Bose--Einstein}		& $(1 + z) \log (1 + z) - z \log(z)$ 	& $[0, \infty)$		
& $\log \left(\dfrac{z}{1+z}\right)$	& 	$-\log(1-e^y)$		& $\dfrac{1}{e^y - 
1}$ \\
		{Fermi--Dirac}			& $(1 - z) \log (1 - z) + z \log(z)$ 	& $[0, 1]$			& 
$\log \left(\dfrac{z}{1-z}\right)$	& 	$\log(1+e^y)$		& $\dfrac{1}{e^y + 1}$ 
\\
		{Quadratic} 			& $\frac12 z^2$ 						& $\bbR$			& 	$z$			& 	
$\frac12 y^2$ 			& $y$
	\end{tabular}
	\caption{Common entropy densities $\eta$.}
	\label{tab:ex-ent}
\end{table}

\subsection{Entropy-based moment methods}

Moment methods encapsulate the velocity-dependence of $f$ in a vector-valued 
function 
\begin{equation}
\bu(t, x)=(u_0(t, x), u_1(t, x), \dots , u_{n - 1}(t, x))) 
\end{equation}
that approximates the velocity averages of $f$ with respect to the vector of 
basis functions 
\begin{equation}
 \bm(v) = (m_0(v), m_1(v), \dots , m_{n - 1}(v)));
\end{equation}
that is, $u_i(t, x) \simeq \vint{m_i f(t, x, \cdot)}$ for all $i \in \{ 0, 1, 
\dots n - 1 \}$.
The components of $\bm$ are typically polynomials and include the collision 
invariants defined in \eqref{eq:conservation}.

The entropy-based moment method is a nonlinear Galerkin discretization in the 
velocity variable.
It has the form
\begin{equation}\label{eq:proj}
 \partial_t  \Vint{\bm F_{\bu}}
  + \nabla_x\cdot \Vint{v \bm  F_{\bu})} = \Vint{\bm 
  \cC(F_{\bu})},
\end{equation}
where $F_{\bu} = F_{\bu(t, x)}(v)$ is an ansatz that
approximates the distribution function $f$ and is consistent with the moment 
vector 
$\bu$.
Unlike the trial function in a traditional (linear) Galerkin method, $F_{\bu}$ 
is not assumed to be a linear combination of the basis functions in $\bm$.
Instead, in an entropy-based moment method, the ansatz is given by the solution 
of a constrained optimization problem whose objective function is 
defined via the kinetic entropy density $\eta$ introduced in the 
previous subsection.  Let
\begin{align}
 \cH(g) := \Vint{\eta(g)}.
\end{align}
Then the defining optimization problem is
\begin{equation}\label{eq:primal}
 \minimize_{g \in \bbF(V)} \: \cH(g)
  \qquad \st \: \Vint{\bm g} = \bv,
\end{equation}
where $\bv \in \bbR^n$ and
\begin{align}
\label{eq:range}
 \bbF(V) = \{ g \in L^1(V) : \text{Range}(g) \subseteq D 
  \}.
\end{align}


\begin{remark}
	Throughout the paper, we reserve the symbol $\bu = \bu(t,x)$ for the solution
	of a partial differential equation like \eqref{eq:proj} 
	(e.g., \eqref{eq:mn} and \eqref{eq:reg-mn} below).
	For a generic moment vector, independent of space and time, we use $\bv$.
	Thus we also use $\bv$ to label the argument of various moment-dependent 
	functions below.
	This deviates somewhat from standard notation but makes many of the 
computations more precise.
\end{remark}

The solution to \eqref{eq:primal}, if it exists,%
\footnote{In general, it may not.  See 
\cite{Jun00,Hauck-Levermore-Tits-2008,caflisch1986equilibrium,
Borwein-Lewis-1991}.
}
takes the form $\Gv$, where
\begin{align}\label{eq:ansatz}
\Ga := \eta'_*(\bsalpha \cdot \bm),
\end{align}
$\alphahat \colon \bbR^n \to \bbR^n$ maps $\bv$ to the solution of the
dual problem 
\begin{equation}\label{eq:dual}
\alphahat(\bv) = \argmax_{\bsalpha \in \R^n} 
			\left\{ \bsalpha \cdot \bv - \Vint{\etad(\bsalpha \cdot \bm)}\right\}
\end{equation} 
and $\etad$ is the Legendre dual%
\footnote{
See, e.g., \cite[\S 3.3.2.]{evans2010partial} or \cite[\S 3.3]{boyd2004convex}, 
where what we call the Legendre dual is called the conjugate function.
}
of $\eta$ (see Table \ref{tab:ex-ent}).
In this case, first-order necessary conditions for \eqref{eq:dual} imply that
\begin{align}\label{eq:dual-grad}
\Vint{\bm \Gv} = \bv.
\end{align}
Hence the function $\vhat \colon \bbR^n \to \bbR^n$ defined by
\begin{align}
\label{eq:vhat}
 \vhata := \Vint{\bm \Ga}
\end{align}
is the inverse of $\alphahat$, 
and the moment equations in \eqref {eq:proj} take the form
\begin{align}\label{eq:mn}
\partial_t \bu + \nabla_x \cdot \bff(\bu) &= \br(\bu),
\end{align}
where the flux function $\bff$ and relaxation term $\br$ are given by
\begin{align}\label{eq:f-and-r}
\bff(\bv) := \Vint{v \bm \Gv} \qquand
\br(\bv) := \Vint{\bm \cC(\Gv)}.
\end{align}

The appeal of the entropy-based approach to closure is that  \eqref{eq:mn} 
inherits many of the structural properties of
the kinetic equation \eqref{eq:kinetic}.
We summarize these here:

\begin{enumerate}[(i)]
	\item \emph{Invariant range}: The natural bounds on the kinetic 
	equation lead to a realizability condition on the solution $\bu$.
	A vector $\bv \in \R^n$ is called
	\emph{realizable (with respect to $\eta$ and $\bm$)} if there exists a
	$g \in \bbF(V)$ such that $\Vint{\bm g} = \bv$.
	The set of all realizable moment vectors is denoted by $\cR$.
	One expects formally that the solution $\bu$ of \eqref{eq:mn} satisfies 
	$\bu(t, x) \in \cR$ for all $(t,x) \in [0,\infty) \times X$.
	If $D = B$, then this means the solution  is always consistent with the 
	bounds on the kinetic density function $f$.
	
	\item \emph{Conservation}: If $m_i \in \bbE$, then
	$r_i(\bv) = \vint{m_i \cC(\Gv)} = 0$ and the $i$-th component of
	\eqref{eq:mn} is 
	\begin{align}\label{eq:conservation-mn}
	\partial_t u_i + \nabla_x \cdot  \Vint{v m_i \Gu} = 0.
	\end{align}
	
	\item \emph{Hyperbolicity \cite{Lev96}}:
	When expressed in terms of $\bsbeta(t,x):=\alphahat(\bu(t,x))$, \eqref{eq:mn}
	takes the form of a symmetric hyperbolic balance law
	\begin{align}\label{eq:mn_symhyp}
	 \hd''(\bsbeta) \partial_t \bsbeta
	  + \jd''(\bsbeta) \cdot \nabla_x \bsbeta &= \br(\vhat(\bsbeta)),
	\end{align}
	where 
	\begin{equation}\label{eq:entropy-entropy-flux-potentials}
	\hd(\bsalpha) = \vint{\etad(\bsalpha \cdot \bm)}
	 \quand
	\jd(\bsalpha) = \vint{v \etad(\bsalpha \cdot \bm)}
	\end{equation}
	are the entropy and entropy-flux potentials, respectively.
	Thus \eqref{eq:mn} is a symmetrizable hyperbolic system.
	
	\item \emph{Entropy dissipation \cite{Lev96}}:
	Assume that \eqref{eq:primal} has a solution for ever vector $\bv$ in the
	image of $\bu$, and let
	\begin{equation}\label{eq:entropy-entropy-flux}
	h(\bv) := \Vint{\eta(\Gv)}
	 \quand  j(\bv) := \Vint{v \eta(\Gv)}
	\end{equation}
	be the entropy and entropy flux, respectively.
	Using the hyperbolic structure of the left-hand side, one can show that $h$ 
	and $j$ are compatible with $\bff$, namely that
	\begin{align}\label{eq:eeflux}
	 j'(\bv) = h'(\bv) \cdot \frac{\partial \bff}{\partial \bv}.
	\end{align}
	Furthermore, we have
	$h'(\bv) \cdot \br(\bv) = \alphahatv \cdot \br(\bv) \le 0$ (where the 
	inequality follows immediately from \eqref{eq:entropy-diss}), and thus
	the moment equations \eqref{eq:mn} inherit a semi-discrete version of the 
	entropy-dissipation law in \eqref{eq:entropy-diss-kin-eq}: 
	\begin{align}\label{eq:entropy-diss-mn}
	\partial_t h(\bu) + \nabla_x \cdot j(\bu) = h'(\bv) \cdot \br(\bv) \le 0.
	\end{align}
	We note that the existence of the entropy and entropy flux pair satisfying
	\eqref{eq:eeflux} is equivalent to symmetric hyperbolicity as in
	\eqref{eq:mn_symhyp}. 
	The dissipation of the right hand side as stated in
	\eqref{eq:entropy-diss-mn}, however, does not translate automatically.
	\item \emph{H-Theorem \cite{Lev96}}:  The H-Theorem for the kinetic equation 
	can be used to show the equivalency of the following statements for 
	\eqref{eq:mn}:
	  \begin{align}
	{\rm (a)}~\alphahatv \cdot \br(\bv) = 0;
	\qquad
	{\rm (b)}~\br(\bv) = 0; 
	\qquad
	{\rm (c)}~\alphahatv \cdot \bm \in \bbE.
	\end{align}
	
	\item \emph{Galilean invariance \cite{JunUnt02}}:  If the kinetic equation is 
	invariant under a transformation $\cG_{O,w}$, defined in \eqref{eq:Galilean}, 
	and if $\operatorname{span}\{m_0,\dots,m_{n-1}\}$ is invariant under 
	$\cG_{O,w}$, then system \eqref{eq:mn} is also invariant under the inherited
	transformation
	\begin{equation}
	\label{eq:mn_Galinv}
	\cT_{O,w} \bu := \vint{ \bm \cG_{O,w} F_{\bu}}.
	\end{equation}
	If we let $T_{O,w}$ be the $n \times n$ matrix satisfying
	$\bm(O(v - w)) = T_{O,w} \bm(v)$,%
	\footnote{
	The subscripts of $T$ are given in the reverse of the order they're
	applied to be consistent with their order in matrix multiplication---i.e., 
	$T_{O, w} = T_{O, 0}T_{I, w}$, where $I$ is the $d \times d$ 
	identity matrix---so that the inverse $(T_{O, w})^{-1}$ is given by
	$T_{-w, O^{-1}} = T_{-w, I}T_{0, O^{-1}}$.
	}
	then we can give $\cT_{O,w}$ explicitly as
	\begin{align}\label{eq:T}
	 (\cT_{O,w} \bu)(t, x) = T^{-1}_{O,w}\bu(t, O(x - tw)).
	\end{align}
	Then the Galilean invariance of \eqref{eq:mn} is reflected by the identity
	\begin{align}\label{eq:mult-identity-gal}
	 \alphahat(T^{-1}_{O,w} \bv) = T^T_{O,w} \alphahatv
	  \qquad \text{(equivalently }
	 T^{-1}_{O,w} \vhata = \vhat(T^T_{O,w} \bsalpha)
	  \text{),}
	\end{align}
	(this can be derived using the first-order necessary conditions
	\eqref{eq:dual-grad})
	as well as the commutability of $\cT_{O,w}$ with the operator
	\begin{align}
	 (\partial_t + \nabla_x \cdot \bff - \br)\bu
	  := \partial_t \bu + \nabla_x \cdot \bff(\bu) - \br(\bu),
	\end{align}
	i.e.,
	\begin{align}\label{eq:mn-gal-inv}
	 (\partial_t + \nabla_x \cdot \bff - \br)(\cT_{O,w} \bu)
	  = \cT_{O,w}((\partial_t + \nabla_x \cdot \bff - \br)\bu).
	\end{align}
\end{enumerate}

\subsection{Realizability and relaxation of the entropy minimization problem}

The realizability condition introduced in the previous subsection can cause 
serious complications for numerical methods.
While it may seem advantageous (for physical reasons) to require that the 
solution in \eqref{eq:mn} be everywhere realizable,
it can unfortunately cause the closure procedure to fail 
rather unforgivingly in numerical simulations.
Specifically, if in the course of a simulation a numerical algorithm generates a 
vector 
$\bv \nin \cR$, then the primal problem \eqref{eq:primal} will be 
infeasible (i.e., the constraint set will be empty) and $\bff(\bv)$ and 
$\br(\bv)$ will not be well-defined.
Discretization errors can easily cause the numerical solution to take 
on values outside of the realizable set, and in such cases, the 
simulation will crash.

Although several algorithms have been designed to maintain the realizability of 
numerical solutions, each has significant limitations.
For example, two kinetic schemes have been proposed: the scheme in 
\cite{AllHau12} is limited to second-order, while the formally higher-order 
method from \cite{SchneiderAlldredge2016} relies on a limiter not rigorously 
shown to preserve accuracy.
Both kinetic schemes have the disadvantage of requiring spatial reconstructions 
for every node of the quadrature in the $v$ variable,%
\footnote{
In practice, the velocity integrals cannot be done analytically, so a 
quadrature is required.
}
and accuracy requirements dictate that there be significantly more nodes than 
moment components \cite{AllHau12}.
Discontinuous-Galerkin schemes have also been considered, but the scheme in 
\cite{Olbrant2012} is limited to first-order moment vectors and one spatial 
dimension, while the limiter used in \cite{AlldredgeSchneider2014} can destroy 
high-order accuracy and relies on an expensive approximate description of $\cR$.
What's more, a deeper problem obstructs the creation of 
realizability-preserving methods: the concrete description of $\cR$ in general 
remains an open problem \cite{LasserreBook}.
Finally, all second- or higher-order methods so far have been limited to 
explicit time integration, which cannot handle the stiffness of the equations 
near fluid-dynamical regimes 
\cite{jin1999ap,mcclarren2008semi,dimarco2013asymptotic} (although the recently 
developed 
algorithm \cite{hu2017asymptotic} may be applicable).

One way to overcome the feasibility issue in \eqref{eq:primal} is to relax the 
constraints.
This is the approach taken in \cite{Decarreau-Hilhorst-Lemarichal-Navaza-1992}, 
where the authors analyzed \eqref{eq:primal} in the context of an inverse 
problem.
Specifically, a function approximation was generated from partially observed 
experimental data that was given by the moment constraints.
Because measurement errors may generate nonrealizeable moments, the authors 
relaxed the equality constraints in \eqref{eq:primal} to arrive at the 
unconstrained problem
\begin{align}\label{eq:tik-primal}
\minimize_{g \in \bbF(V)} \:\cH_\gamma(g; \bv),
\end{align}
with the modified objective function
\begin{align}
\label{eq:H-gamma}
\cH_\gamma(g; \bv) := \Vint{\eta(g)} + \frac1{2\gamma}
\left\| \Vint{\bm g} - \bv \right\|^2.
\end{align}
Here $\gamma \in (0, \infty)$ is a parameter and $\|\cdot\|$ 
is the usual Euclidean norm on $\R^n$.
Unlike the original primal problem \eqref{eq:primal}, the relaxed problem 
\eqref{eq:tik-primal} is feasible for \emph{any} $\bv \in \R^n$ (not just 
$\bv \in \cR$), so we expect that it will have a solution for most, indeed 
perhaps all, $\bv \in \R^n$.

Whenever a solution to \eqref{eq:tik-primal} exists, it has the same 
form as that of the original primal problem:
\begin{align}
\argmin_{g \in \bbF(V)} \left\{ \cH_\gamma(g; \bv) \right\} =
\Ggv, 
\end{align}
where $\Ga$ is defined in \eqref{eq:ansatz} and $\alphahatg(\bv)$ is the
solution of the new dual problem:
\begin{align}\label{eq:reg-mult}
\alphahatgv := \argmax_{\bsalpha \in \R^n} \left\{
\bsalpha \cdot \bv
- \Vint{\eta_*(\bsalpha \cdot \bm)}
- \frac\gamma2 \|\bsalpha\|^2 \right\}.
\end{align}
Thus the relaxation of the constraints in the primal corresponds to a Tikhonov 
regularization of the dual 
\cite{Decarreau-Hilhorst-Lemarichal-Navaza-1992}.
For this reason, we refer to $\gamma$ as the regularization parameter.
Indeed, the condition number of the Hessian of the dual objective in 
\eqref{eq:reg-mult} is bounded from above by $1 + \gamma^{-1} c $, where $c$ is 
the maximum eigenvalue of the Hessian of the original dual function 
\eqref{eq:dual}; this bound decreases as 
$\gamma$ increases.
The regularization provided by $\gamma$ can be helpful for vectors $\bv \in 
\cR$ near the boundary of $\cR$, when the original dual problem 
\eqref{eq:dual} can be difficult to solve \cite{AllHau12}.

The price to pay for relaxing the constraints in \eqref{eq:primal} is the 
mismatch between $\vint{\bm \Ggv}$ and $\bv$; that is, unlike 
\eqref{eq:dual-grad}, $\vint{\bm \Ggv} \ne \bv$.
However, because of measurement or simulation errors, $\bv$ is not known 
precisely in practice anyway; nor can the dual problem \eqref{eq:dual} be 
solved exactly.
Hence if $\gamma$ is sufficiently small, then overall accuracy can be 
maintained.
This statement can be quantified more precisely using the following definition 
and theorem.
\begin{defn}
	\label{defn:tau-optimal}
	Let $\tau>0$.
	Then
	\begin{align}
	\bG_\gamma^\tau(\bv) := \left\{ g^* \in \bbF(V) : \cH_\gamma(g^*; \bv) \leq
	\inf_{g \in \bbF(V)} \left\{ \cH_\gamma(g; \bv) \right\} + \tau \right\}
	\end{align}
	is the set of all \emph{$\tau$-optimal} density functions.
\end{defn}

\begin{thm}[\!\!\cite{engl1989convergence,engl1993convergence}]
	\label{thm:acc}
	Let $\bv^\delta$ be a moment vector satisfying
	$\|\bv - \bv^\delta\| \le \delta$ for some $\bv \in \cR$ and
	$g \in \bG_\gamma^\tau(\bv^\delta)$.
	If $\gamma \sim \delta$ {\rm (}i.e., $\gamma = \cO(\delta)$ and
	$\delta = \cO(\gamma)${\rm )} and $\tau = \cO(\delta)$, then
	\begin{align}\label{eq:reg-err}
	\left\| \Vint{\bm g} -  \bv \right\| = \cO(\delta).
	\end{align}
\end{thm}
\noindent Theorem \ref{thm:acc} provides a strategy for choosing  $\gamma$ (and 
$\tau$) so that the regularized problem can be used to solve \eqref{eq:mn} 
without 
losing the order of accuracy.

%% file: sections/rmn.tex
\section{Regularized entropy-based closures}\label{sec:structure}

In this section, we propose a new set of closures, 
based on the regularization \eqref{eq:tik-primal}.
We replace \eqref{eq:mn} by the system
of regularized entropy-based moment equations
\begin{align}\label{eq:reg-mn}
\partial_t \bu + \nabla_x \cdot \bff_\gamma(\bu) &= \br_\gamma(\bu),
\end{align}
where (cf. \eqref{eq:f-and-r}) 
\begin{align}\label{eq:f-and-r-gam}
\bff_\gamma(\bv) := \Vint{v \bm \Ggv} \qquand
\br_\gamma(\bv) := \Vint{\bm \cC(\Ggv)}
\end{align}
are defined even when $\bv \nin \cR$.
The system \eqref{eq:reg-mn} can then used to approximate the original system 
\eqref{eq:mn} numerically without having to enforce realizability conditions 
explicitly.

In the remainder of the section, we examine the structural properties of the 
system of regularized moment equations \eqref{eq:reg-mn}.
For most of this section (particularly in Sections \ref{sec:mult-mom} and 
\ref{sec:rmn-structure}) we assume the primal problem \eqref{eq:tik-primal} has 
a minimizer.
While this assumption is necessary to rigorously justify many of the formal 
calculations that follow in this section, there are important cases for which 
it 
does not hold.
These exceptions are subject of \secref{junk-preview} and the Appendix. 
Under this assumption we use Legendre duality to establish the formal 
relationship between a moment vector $\bv$ and its corresponding multiplier 
vector $\alphahatgv$.
Then, as in \cite{Lev96}, this relationship allows us to investigate the
structure of the regularized system \eqref{eq:reg-mn}.

\subsection{Regularized moment-multiplier relationship}
\label{sec:mult-mom}

Many of the structural properties of \eqref{eq:mn} rely on duality relations,
which we establish here for the regularized case.
We first define the convex function ${\hg : \bbR^n \to \R}$ by
\begin{align}\label{eq:hg}
 \hg(\bv) := \inf_{g \in \bbF(V)} \left\{ \cH(g) + \frac1{2\gamma}
   \| \Vint{\bm g} - \bv \|^2 \right\}.
\end{align}
First-order optimality conditions for the dual \eqref{eq:reg-mult} imply that
\begin{align}\label{eq:reg-dual-grad}
\bv = \Vint{\bm \Ggv} + \gamma \alphahatgv.
\end{align}
From \eqref{eq:reg-dual-grad}, we conclude that
\begin{align}\label{eq:vhatg}
\vhatga := \vhat(\bsalpha)+ \gamma \bsalpha,
\end{align}
where $\vhat$ is defined in \eqref{eq:vhat}, is the inverse of $\alphahatg$.
When $\gamma= 0$, we recover the original moment map:
\begin{align}
 \vhatga\big|_{\gamma = 0} 
  	= \vhat(\bsalpha).
\end{align}
Furthermore, under the assumption that the infimum in \eqref{eq:hg} 
is attained, substitution of \eqref{eq:reg-dual-grad} into \eqref{eq:hg} gives
\begin{equation}
\hg(\bv)
	= \cH(\Ggv)
		+ \frac\gamma 2 \| \alphahatgv \|^2
	= h(\vhat(\alphahatgv)) + \frac\gamma 2 \| \alphahatgv \|^2
\end{equation}
Thus when $\bv$ is realizable, $\hg \to h$ (cf. 
\eqref{eq:entropy-entropy-flux}) as $\gamma \to 0$.%
\footnote{
This limit follows directly since (i) $h$ and $\vhat$ are continuous functions 
and (ii) $\alphahatg$ is continuous with respect to $\gamma$ for
$\gamma \in [0, \infty)$ when $\bv \in \cR$.
Property (ii) follows from the same continuity of $\vhatg$, the inverse of
$\alphahatg$.
}

Duality relations established in \cite{Borwein-Lewis-1991} imply that
$\hg(\bv)$ is equal the maximum of the regularized dual problem 
\eqref{eq:reg-mult}.
Therefore $\hg$ is, by definition, the Legendre dual of
the convex function ${\hgd \colon \R^n \to \R}$, defined by
\begin{align}\label{eq:hgd}
 \hgd(\bsalpha) := \Vint{\etad(\bsalpha \cdot \bm)}
  + \frac\gamma2 \|\bsalpha\|^2.
\end{align}
Differentiating this formula gives $\hgd'(\bsalpha) = \vhatga$. 
According to the theory of Legendre duality $(\hgd')^{-1} = \hg'$, so from 
\eqref{eq:reg-dual-grad} we have
\begin{align}
 \hg'(\bv) = \alphahatgv.
\end{align}

The Hessian matrices are now straightforwardly computed:
\begin{align}
 \hgd''(\bsalpha) &= \frac{\partial \vhatg}{\partial \bsalpha}
   = \Vint{\bm \bm \cdot \etad''(\bsalpha \cdot \bm)} + \gamma I
   =: H_\gamma(\bsalpha), \quad \text{and} \label{eq:hess-g} \\
 \hg''(\bv) &= \frac{\partial \alphahatg}{\partial \bv}
   = H^{-1}_\gamma(\alphahatgv), \label{eq:hess-inv-g}
\end{align}
where $I$ is the $n \times n$ identity matrix.

\subsection{Structural properties of the regularized equations}
\label{sec:rmn-structure}

The duality relations from the last section now allow us to check whether the 
regularized moment system inherits the structural properties of the underlying 
kinetic equation.

\begin{enumerate}[(i)]
	\item \emph{Invariant range}: While the regularized equations are defined 
	even for nonrealizable moment vectors, the underlying ansatz $\Ggv$ used in 
	the flux and collision terms takes on the same range of values as the 
	original entropy ansatz.
	
	\item \emph{Conservation}: If $m_i \in \bbE$, then
	$r_{\gamma,i}(\bv) = \vint{m_i \cC(\Ggv)} = 0$ and the $i$-th component of
	\eqref{eq:mn} is 
	\begin{align}\label{eq:conservation-rmn}
	\partial_t u_i + \nabla_x \cdot  \Vint{v m_i \Ggu} = 0.
	\end{align}
	
	\item \emph{Hyperbolicity}:
	When expressed in terms of $\bsbeta(t, x) := \alphahatg(\bu(t, x))$, 
	\eqref{eq:reg-mn} takes the form of a symmetric hyperbolic balance law
	\begin{align}\label{eq:rmn_symhyp}
	 \hgd''(\bsbeta) \partial_t \bsbeta
	  + \jd''(\bsbeta) \cdot \nabla_x \bsbeta &= \br_\gamma(\bu),
	\end{align}
	where $\jd$ is the original entropy-flux potential (see 
	\eqref{eq:entropy-entropy-flux-potentials}).
	Thus \eqref{eq:reg-mn} is also a symmetrizable hyperbolic system.
	
	\item \emph{Entropy dissipation}:  
	With the original entropy flux in mind, we define
	\begin{align}
	\jg(\bv) := \Vint{v \eta(\Ggv)}.
	\end{align}
	Then $\hg$ and $\jg$ are compatible with $\bff_\gamma$, i.e.,
	\begin{align}
	 j'(\bv) = h'(\bv) \cdot \frac{\partial \bff}{\partial \bv},
	\end{align}
	and we also have
	$\hg'(\bv) \cdot \br(\bv) = \alphahatgv \cdot \br_\gamma(\bv) \le 0$ from
	\eqref{eq:entropy-diss}.
	Thus the regularized moment equations \eqref{eq:reg-mn} have the
	entropy-dissipation law
	\begin{align}\label{eq:entropy-diss-rmn}
	\partial_t \hg(\bu) + \nabla_x \cdot \jg(\bu) = \hg(\bu) \cdot \alphahatgu
	\le 0.
	\end{align}
	\item \emph{H-Theorem}:  Just as with the original equations, the 
	following statements are equivalent:
	\begin{align}
	 {\rm (a)}~\alphahatgv \cdot \br_\gamma(\bv) = 0;
	 \qquad
	 {\rm (b)}~\br_\gamma(\bv) = 0; 
	 \qquad
	 {\rm (c)}~\alphahatgv \cdot \bm \in \bbE.
	\end{align}
	However, the moment vectors $\bv$ satisfying these conditions may not be 
	the same as those of the original system, i.e.,
	$\br_\gamma^{-1}(0) \ne \br^{-1}(0)$.
	\item \emph{Galilean invariance}:
	In order to take advantage of the Galilean invariance of the original 
	equations, we use the identity
	$\bu = \vhat(\alphahatgu) + \gamma \alphahatgu$ and write the regularized 
	equations as
	\begin{align}\label{eq:reg-mn-0-lhs}
	 0 = \gamma \partial_t \alphahatgu
	  + (\partial_t + \nabla_x \cdot \bff - \br)(\vhat(\alphahatgu)).
	\end{align}
	It turns out that we must consider rotations and velocity translations 
	separately.
	Let's first consider the rotation $\cT_{O,0}$.
	Note that if the matrix $T_{O,0}$ (recall \eqref{eq:T}) is orthogonal, we have
	$\alphahatg(T^{-1}_{O,0}\bv) = T^{-1}_{O,0} \alphahatgv$ (from the 
	first-order necessary conditions \eqref{eq:reg-dual-grad}) and thus
	$\partial_t \alphahatg(\cT_{O,0}\bu) = \cT_{O,0} \partial_t \alphahatgu$.
	When we combine this with \eqref{eq:mult-identity-gal} and
	\eqref{eq:mn-gal-inv}, we have
	\begin{subequations}
	\begin{align}
	 0 &= \gamma \partial_t \alphahatg(\cT_{O,0} \bu)
	    + (\partial_t + \nabla_x \cdot \bff - \br)
	    (\vhat(\alphahatg(\cT_{O,0} \bu))) \\
	   &= \cT_{O,0}(\gamma \partial_t \alphahatgu
	    + (\partial_t + \nabla_x \cdot \bff - \br)(\vhat(\alphahatgu))),
	\end{align}
	\end{subequations}
	which shows that the regularized moment system is rotationally invariant.
	One can show that the matrix $T_{O,0}$ is indeed orthogonal if there exists a 
	radially symmetric weight function $\omega = \omega(v)$ so that
	$\vint{\bm \bm^T \omega} = I$, i.e., so that the basis functions are 
	orthonormal with respect to $\omega$.%
	\footnote{
	We show this using $\vint{\bm \bm^T \omega} = I$ and computing
	\begin{gather}\label{eq:T-orth-rot}
	 T^{-1}_{O,0} = T^{-1}_{O,0} \Vint{\bm \bm^T \omega}
	  = \Vint{\bm(O^{-1}v) \bm^T \omega}
	  = \Vint{\bm(v) (\bm(O v))^T \omega(|O v|)} \nonumber \\
	  = \Vint{\bm (T_{O,0}\bm)^T \omega(|v|)}
	  = T_{O,0}^T.
	\end{gather}
	(Here we use $|\cdot|$ for the Euclidean norm on $\R^d$ and reserve
	$\|\cdot\|$ for the Euclidean norm for moment vectors.)
	This orthonormality assumption holds, e.g., for the normalized spherical 
	harmonics on the unit sphere.
	}
	However, for a velocity translation $\cT_{I,w}$ we have
	\begin{align}
	 \partial_t \alphahatg(\cT_{I,w}\bu)
	  = \frac{\partial \alphahatg}{\partial \bv}\left(T_{I,w}\left((\partial_t \bu
	   + w \cdot \nabla_x \bu)\big|_{(t, x - tw)} \right) \right).
	\end{align}
	Even if $T_{I,w}$ is orthogonal, the additional $w \cdot \nabla_x \bu$ term 
	is neither part of $\cT_{I,w} \partial_t \alphahatgu$ nor is it canceled by 
	anything else in the right-hand side of \eqref{eq:reg-mn-0-lhs}.
	Thus the regularized equations fail to be translation invariant.
\end{enumerate}

\subsection{Degenerate densities}
\label{sec:junk-preview}

One of the major drawbacks of entropy-based moment closures is that there exist 
realizable moment vectors $\bv$ for which the original primal problem 
\eqref{eq:primal} has no solution.
For these \emph{degenerate densities}, many of the structural properties of the 
entropy-based formulation are lost.
The geometry of these densities was investigated in detail for the 
Maxwell-Boltzmann entropy \cite{Junk-1998}, with $V = \R$ and
${\bm(v) = (1, v, v^2, v^3, v^4)}$; extensions to multiple dimensions and more 
general polynomial basis functions can be found in 
\cite{Jun00,schneider2004entropic,Hauck-Levermore-Tits-2008}.

Unfortunately, the regularization does not fix the problem of degeneracy.
Indeed, there are also moment vectors $\bv$ for which the regularized primal 
problem \eqref{eq:tik-primal} does not achieve its minimum.
For the original primal, the fundamental issue is that for a fixed $\bv$ the 
constraint set $\{g \in \bbF(V) : \vint{\bm g} = \bv \}$ is not closed when $V$ 
is unbounded, in particular when $V = \R^d$, because the map
$g \mapsto \vint{\bm g}$ is not continuous.
This issue carries over to the regularized problem, since this discontinuous map 
appears in the objective function $\cH_\gamma$, so that $\cH_\gamma$ is not
lower-semicontinuous.

Although not exactly the same, the set of degenerate moment vectors for the
regularized problem can be characterized in the same fashion as the degenerate 
moment vectors for the original problem.
As an illustrative example, consider the Maxwell--Boltzmann entropy with
$V = \R^d$ and $m_{n - 1}(v) = |v|^N$, where $m_{n - 1}$ is the only component 
of $\bm$ with degree greater than or equal to $N$.
(This includes the example mentioned above from \cite{Junk-1998}).
Let $\cA$ be the set of multiplier vectors such that $\Ga \in L^1(V)$.
Then the main result of \cite{Jun00} can be extended to the following:

\begin{prop}\label{prop:junk-line-g}
If $\bv$ can be written as
\begin{align}\label{eq:junk-form-g}
 \bv = \vhatg(\alphabar) + \begin{pmatrix}
                            0 \\ \vdots \\ 0 \\ \delta
                           \end{pmatrix},
\end{align}
for some $\alphabar \in \cA \cap \partial \cA$ and $\delta \in (0, \infty)$,
then $\bv$ is a degenerate density for the regularized problem, i.e., 
$\cH_\gamma(\cdot ; \bv)$ does not achieve its minimum.
\end{prop}

The results from \cite{Junk-1998,Jun00} are recovered when $\gamma = 0$.
We postpone a proof and further discussion to the Appendix.

\subsection{Examples}\label{sec:ex}

Now we take a look at how the regularization affects the most well-known 
instances of the entropy-based moment method.
The simplest case is the P$_N$ equations of radiation transport. 
For the case of bounded velocity domains, we also consider the M$_1$ equations.
For the case of unbounded velocity domains, we study the Euler equations.
For the latter two, we only consider the one-dimensional cases for simplicity.
To make some computations feasible, we define a partially regularized version of
\eqref{eq:primal}:
	\begin{subequations}
		\label{eq:p-tik-primal}
		\begin{align}
		\minimize_{g \in \bbF(V)} \quad & \Vint{\eta(g)}
		+ \frac1{2\gamma}
		\sum_{i = m + 1}^{n - 1} \left( \Vint{m_i g} - v_i \right)^2, \\
		\st \quad & \Vint{m_i g}
		= v_i, \quad i \in \{ 0, 1, \dots , m \},
		\label{eq:u0-constraint}
		\end{align}
	\end{subequations}
	with dual problem
	\begin{align}
	\maximize_{\bsalpha \in \R^n}\:
	\bsalpha \cdot \bv - \Vint{\eta_*(\bsalpha \cdot \bm)}
	- \frac\gamma2 \sum_{i = m + 1}^{n - 1} \alpha_i^2.
	\end{align}
	For the existence of a solution to the primal and dual problems, the
	subvector $(v_0, v_1, \dots , v_m)$ must of course satisfy realizability
	conditions.

\subsubsection{Regularized P$_N$ equations}
\label{sec:reg-pn}

Consider as velocity domain the unit sphere $V = S^2$, and the spherical
harmonics as basis functions.
The P$_N$ equations are an entropy-based closure with the entropy density
$\eta(g) = \frac12 g^2$. This function is equal to its Legendre dual,
$\eta = \eta_*$. 

The unregularized multipliers satisfy $\Vint{\bm\bm^T} \alphahatv =\bv$.
Since the spherical harmonics are an orthonormal basis, i.e.,
$\Vint{ \bm \bm^T} = I$, the ansatz is $G_{\alphahatv} = \bm \cdot \bv$.
The regularized multipliers satisfy
\begin{align}
 (\gamma I+\Vint{\bm\bm^T}) \alphahatgv =\bv,
\end{align}
so 
$G_{\alphahatgv} = \frac{1}{1+\gamma}\bm\cdot \bv$.
This leads to
\begin{align}
 \bff_\gamma(\bv)=\frac{1}{1+\gamma} \bff(\bv).
\end{align}
Hence the regularization acts as a filter \cite{mcclarren2010simulating,mcclarren2010robust} that damps the flux of the original equations.

\subsubsection{Regularized M$_1$ equations}

Here the velocity domain is $V = [-1, 1]$ (i.e., the one-dimensional 
slab-geometry setup) and $\bm(v) = (1, v)$.
The realizable set is given by ${\cR = \{(v_0, v_1) \in \R^2 : |v_1| < v_0 \}}$.
We consider the Maxwell--Boltzmann entropy.%
\footnote{
The M$_1$ method is also often applied to the gray equations for photon 
transport using the Bose--Einstein entropy.
These equations have the advantage that the flux $\bff$ can be given 
analytically \cite{Dubroca-Feugas-1999}.
Unfortunately, this property is (as far as we can tell) destroyed by the 
introduction of $\gamma$, so we do not discuss this particular example in 
further detail.%
}

While no analytical expression can be obtained for the multipliers, one can 
eliminate the zero-th order multiplier $\hat \alpha_0(v_0, v_1)$ so that the 
optimal first-order multiplier $\hat \alpha_1(v_0, v_1)$ satisfies the single 
equation \cite{Min78,BruHol01}
\begin{align}\label{eq:alpha1}
 \frac{v_1}{v_0} = \coth(\hat \alpha_1) - \frac1{\hat \alpha_1},
\end{align}
where for clarity of exposition we suppress the dependence of the optimal 
multipliers on the moment components.
The map $\alpha_1 \mapsto \coth(\alpha_1) - 1 / \alpha_1$ is indeed a 
smooth bijection between $\R$ and $(-1, 1)$, which is consistent with the
existence and uniqueness of the multipliers for $(v_0, v_1) \in \cR$.

We have been unable to decouple the equations for $\hat \alpha_{\gamma, 0}$ and 
$\hat \alpha_{\gamma, 1}$ when regularization is applied to both moment 
components.
However, when we only regularize the first-order moment, i.e., when we solve
\begin{subequations}
\begin{align}
 v_0 &= \Vint{\exp(\hat \alpha_{\gamma, 0} + \hat \alpha_{\gamma, 1} \mu)}
  = \frac2{\hat \alpha_{\gamma, 1}} \exp(\hat \alpha_{\gamma, 0})
   \sinh(\hat \alpha_{\gamma, 1}), \\
 v_1 &= \Vint{\mu \exp(\hat \alpha_{\gamma, 0} + \hat \alpha_{\gamma, 1} \mu)}
   + \gamma \hat \alpha_{\gamma, 1}1 \nonumber \\
  &= \frac2{\hat \alpha_{\gamma, 1}} \exp(\hat \alpha_{\gamma, 0})
   \left( \cosh(\hat \alpha_{\gamma, 1})
   + \frac{\sinh(\hat \alpha_{\gamma, 1})}{\hat \alpha_{\gamma, 1}}  \right)
   + \gamma \hat \alpha_{\gamma, 1},
\end{align}
\end{subequations}
then we can again isolate $\hat \alpha_{\gamma, 1}$ to get
\begin{align}\label{eq:alpha1gam}
 \frac{v_1}{v_0} = \coth(\hat \alpha_{\gamma, 1})
  - \frac1{\hat \alpha_{\gamma, 1}} + \frac\gamma{v_0} \hat \alpha_{\gamma, 1}.
\end{align}
The map
$\alpha_1 \mapsto \coth(\alpha_1) - 1 / \alpha_1 + \gamma \alpha_1 / v_0$
is a smooth bijection from $\R$ to $\R$---under the assumption $v_0 > 0$ (which 
is necessary for the existence of a minimizer in the partially regularized 
case).
Thus the partially regularized problem has a solution for
$(v_0, v_1) \in \{ (v_0, v_1) : v_0 > 0 \} \supset \cR$.
\figref{m1alpha1} plots the maps \eqref{eq:alpha1} and \eqref{eq:alpha1gam}.

\begin{figure}
\label{fig:comparison}
\begin{subfigure}[t]{0.48\linewidth} 
\includegraphics[width=\linewidth]{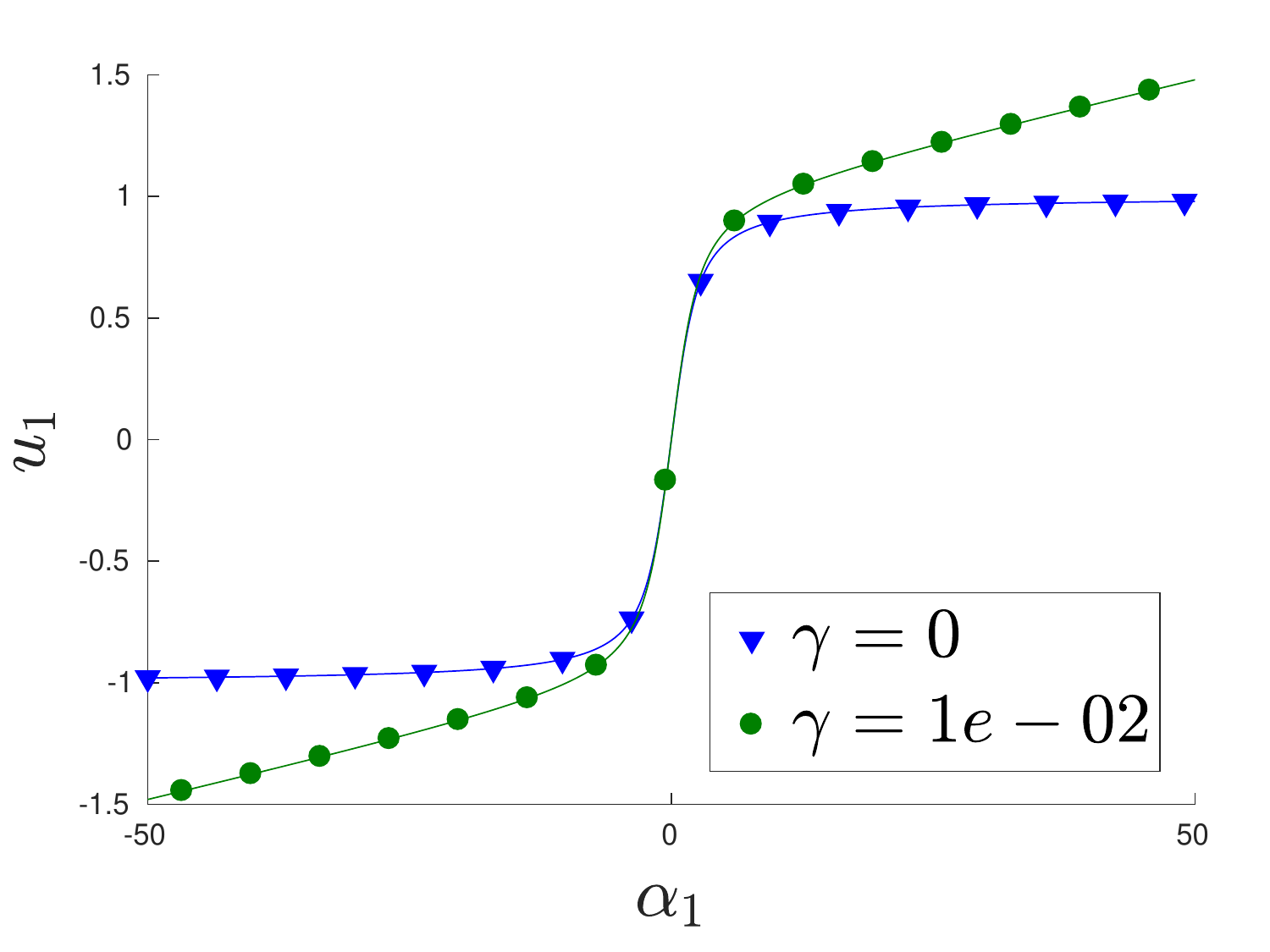}
\subcaption{Comparison of the maps \eqref{eq:alpha1} and \eqref{eq:alpha1gam} 
  which relate the optimal first-order multiplier to the moment components
  for M$_1$ with Maxwell--Boltzmann statistics with $v_0 = 1$.
   \label{fig:m1alpha1}
}
\end{subfigure}
\hfill
\begin{subfigure}[t]{0.48\linewidth} 
 \includegraphics[width=\linewidth]{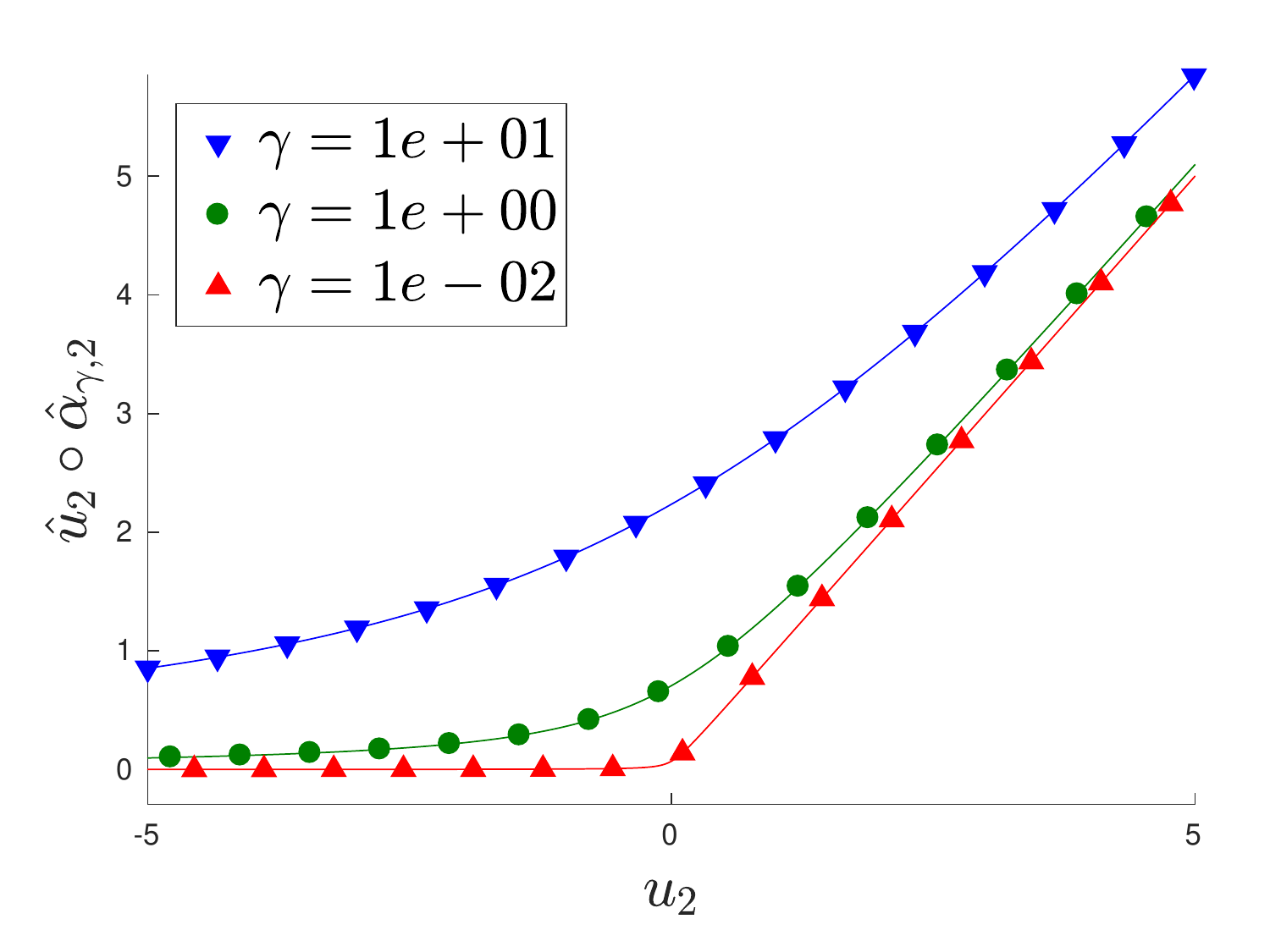}
 \subcaption{Comparison of the regularized closures for $v_2$ of the Euler
 equations, given by \eqref{eq:reg-euler}.
 Here $v_0 = 1$ and $v_1 = 0$.}
  \label{fig:u2g}
\end{subfigure}
\caption{Comparison of regularized versus unregularized closures.}
\end{figure}

%
%
%
%

\subsubsection{Regularized Euler equations}

With $V = \R^3$ and $\bm(v) = \{1, v, |v|^2\}$, the original entropy-based 
moment equation gives the compressible Euler equations \cite{Lev96}.
In one-dimension, i.e., $V = \R$ and ${\bm(v) = (1, v, v^2)}$.
The realizable set is $\cR = \{ (v_0, v_1, v_2) \in \R^3 : v_0 v_2 > v_1^2 \}$.
The moment and optimal multiplier components satisfy
\begin{subequations}
\begin{align}
 v_0 &= \sqrt{-\frac\pi{\hat \alpha_2}} \exp \left( \hat \alpha_0
  - \frac{\hat \alpha_1^2}{4 \hat \alpha_2} \right), \label{eq:a0euler} \\
 v_1 &= \sqrt{-\frac\pi{\hat \alpha_2}} \exp \left( \hat \alpha_0
  - \frac{\hat \alpha_1^2}{4 \hat \alpha_2} \right)
  \frac{-\hat \alpha_1}{2\hat \alpha_2}, \label{eq:a1euler} \\
 v_2 &= \sqrt{-\frac\pi{\hat \alpha_2}} \exp \left( \hat \alpha_0
  - \frac{\hat \alpha_1^2}{4 \hat \alpha_2} \right)
  \left( \frac{\hat \alpha_1^2}{4\hat \alpha_2^2}
  - \frac1{2 \hat \alpha_2}\right), \label{eq:a2euler}
\end{align}
\end{subequations}
and one can readily invert these equations.

Again, we have been unable to solve these equations analytically when all 
moment components are regularized.
We were only able to find an analytical solution for the case when we only 
relax the equality constraint on $v_2$.
In this case \eqref{eq:a2euler} becomes
\begin{align}
 v_2 = \sqrt{-\frac\pi{\hat \alpha_{\gamma, 2}}}
  \exp \left( \hat \alpha_{\gamma, 0}
  - \frac{\hat \alpha_{\gamma, 1}^2}{4 \hat \alpha_{\gamma, 2}} \right)
  \left( \frac{\hat \alpha_{\gamma, 1}^2}{4\hat \alpha_{\gamma, 2}^2}
  - \frac1{2 \hat \alpha_{\gamma, 2}} \right) + \gamma \hat \alpha_{\gamma, 2},
\end{align}
and with appropriate substitutions of \eqref{eq:a0euler}--\eqref{eq:a1euler}
(with the multipliers now labeled with $\gamma$), we get
\begin{align}
 \hat \alpha_{\gamma, 2} = \frac{v_2 v_0 - v_1^2
  - \sqrt{(v_2 v_0 - v_1^2)^2 + 2\gamma v_0^3}}{2\gamma v_0}.
\end{align}
Then the regularized second-order moment becomes:
\begin{align}
\label{eq:reg-euler}
 \hat v_2(\alphahatgv) = v_0 \left(
  \frac{v_1^2}{v_0^2} + \frac{v_2 v_0 - v_1^2
  + \sqrt{(v_2 v_0 - v_1^2)^2 + 2\gamma v_0^3}}{2v_0^2} \right).
\end{align}
\figref{u2g} shows this relationship, and it behaves as expected: when $\gamma$ 
approaches zero, it approaches the identity map for positive $v_2$;
otherwise, for nonphysical negative values of $v_2$, the map returns small 
positive values which get even smaller as $\gamma$ goes to zero.

%% file: sections/accuracy.tex
\section{Accuracy of the closure}
\label{sec:acc}

While the properties in Section \ref{sec:rmn-structure} provide basic structure
of the regularized entropy-based moment equations \eqref{eq:reg-mn}, 
\thmref{acc} hints at an attractive possible application: the use the 
regularized system to accurately solve the \emph{original} moment system 
\eqref{eq:mn}.
To explore this idea further, we note that
$\bff_\gamma(\bv) = \bff (\vhat(\alphahatgv) $, where 
$\vhat(\alphahatgv) \in \cR$ is the regularization of $\bv \in \bbR^n$.
(Indeed, the term $\vint{\bm g}$ in \eqref{eq:reg-err} is 
an approximate evaluation of $\vhat \circ \alphahatg$ at $\bv^\delta$.)
Thus under the assumption that the Jacobian of $\bff$ is bounded, the moment 
mismatch $\vhat(\alphahatgv) - \bv$ can be used to estimate
$\bff_\gamma(\bv) - \bff(\bv)$.
If the $\cO(\delta)$-accuracy in \thmref{acc} holds uniformly for all
$\bv \in \cR$, then we can see $\bff_\gamma$ as way to approximately, 
but accurately evaluate $\bff$.
The collision term $\br_\gamma$ can be considered similarly.

\subsection{Accuracy of the moment regularization map}

With the help of the relationships in \secref{mult-mom}, we can now analyze the 
moment regularization map $\vhat \circ \alphahatg$ directly.

\begin{thm}\label{thm:acc-new-no-tau}
Let
\begin{align}\label{eq:RM}
 \bv \in \cR^M := \{ \bv : \|\alphahatv\| < M \}, 
\end{align}
and let $\bv^\delta$ 
satisfy
\begin{align}\label{eq:udelta}
 \|\bv^\delta - \bv\| \le \delta.
\end{align}
Then
\begin{align}\label{eq:proj-error}
 \|\vhat(\alphahatg(\bv^\delta)) - \bv\| \le \delta + M\gamma.
\end{align}
\end{thm}

\begin{proof}
Let $\vtilde := \vhatg(\alphahatv)$. Then $\bv = \vhat(\alphahatg(\vtilde))$ so 
that
\begin{subequations}\label{eq:udiff-int}
\begin{align}
 \left\|\vhat(\alphahatg(\bv^\delta)) - \bv \right\|
  &= \left\|\vhat(\alphahatg(\bv^\delta)) - \vhat(\alphahatg(\vtilde))
   \right\| \\
  &= \left\|\int_0^1 \left.
   \frac{\partial (\vhat \circ \alphahatg)}{\partial \bv}
   \right|_{\vtilde + s (\bv^\delta - \vtilde)} (\bv^\delta - \vtilde) \intds
   \right\| \\
  &= \left\|\int_0^1 \left. H(\alphahatg)(H(\alphahatg) + \gamma I)^{-1}
   \right|_{\vtilde + s (\bv^\delta - \vtilde)} (\bv^\delta - \vtilde) \intds
   \right\| \\
  &\le \int_0^1 \left\| \left. 
   H(\alphahatg)(H(\alphahatg) + \gamma I)^{-1}
   \right|_{\vtilde + s (\bv^\delta - \vtilde)} \right\| \intds
   \|\bv^\delta - \vtilde\| \\
  &= \int_0^1 \left. 
   \frac{\| H(\alphahatg) \|}{\| H(\alphahatg) \| + \gamma}
   \right|_{\vtilde + s (\bv^\delta - \vtilde)} \intds
   \|\bv^\delta - \vtilde\| \\
  &\le \|\bv^\delta - \vtilde\|.
\end{align}
\end{subequations}
The inverse relationship between $\vhat$ and $\alphahat$, along with 
\eqref{eq:vhatg}, gives
\begin{align}
\label{eq:utilde-alpha}
\bv^\delta - \vtilde 
	= (\bv^\delta - \bv) + (\bv - \vtilde )
	&\stackrel{\hphantom{\eqref{eq:vhatg}}}{=} (\bv^\delta - \bv)
	 + (\vhat(\alphahatv) - \vhatg(\alphahatv)) \nonumber \\
	&\stackrel{\eqref{eq:vhatg}}{=} (\bv^\delta - \bv) -  \gamma \alphahatv).
\end{align}
Altogether we have
\begin{align}
 \| \vhat(\alphahatg(\bv^\delta)) - \bv \| 
  \stackrel{\eqref{eq:udiff-int}}{\le} \|\bv^\delta - \vtilde\| 
  \stackrel{\eqref{eq:utilde-alpha}}{\le} 
   \|\bv^\delta - \bv \| + \gamma \|\alphahatv\| 
  &\stackrel{\eqref{eq:udelta}}{\le} \delta + \gamma \|\alphahatv\| \nonumber \\
  &\stackrel{\eqref{eq:RM}}{\le} \delta + M\gamma.
\end{align}
\ 
\end{proof}

As a result of Theorem \ref{thm:acc}, if $\gamma \le C \delta$ for some
$C \in (0, \infty)$, then the moment regularization error 
$\|\vhat(\alphahatg(\bv^\delta)) - \bv\|$ is $\cO(\delta)$.  This result gives 
uniform accuracy over a large set of moment vectors, but only
by bounding this set away from the boundary of the realizable set by 
controlling the norm of the associated multiplier vectors.
The closer the moment vectors get to the boundary of the realizable 
set, the larger the constant in the error bound becomes.

\subsection{Stopping criterion for the optimization}
\label{sec:stopping}

As in \thmref{acc}, we would like to use an approximate solution to the 
optimization problem while keeping $\cO(\delta)$ accuracy.
The error is quantified by the value of the primal objective 
function, but previous work with entropy-based moment methods has used 
the norm of the dual gradient for the stopping criterion (e.g., 
\cite{AllHau12}).
The norm of the dual gradient is preferable because it is already computed 
by the optimizer (as part of the search-direction computation) and is 
easy to interpret.
We will show that these two stopping criteria are closely related, but first we 
give our result with the gradient-based criterion.

\begin{thm}
Assume $\bv \in \cR^M$; let $\bv^\delta$ satisfy
$\|\bv^\delta - \bv\| \le \delta$; and let $\bsalpha$ satisfy
\begin{align}\label{eq:stopping}
 \| \vhat(\bsalpha) - \bv^\delta + \gamma \bsalpha \| \le \tau.
\end{align}
Then
\begin{align}\label{eq:proj-error-tau}
 \|\vhat(\bsalpha) - \bv\| \le 2\delta + M\gamma + 2\tau.
\end{align}
\end{thm}

\begin{proof}
We write 
\begin{equation}
\vhat(\bsalpha) - \bv 
	= \vhat(\bsalpha) - \bv^\delta + \gamma \bsalpha
	+ (\bv^\delta - \bv)
	- \gamma \bsalpha.
\end{equation}
and apply the triangle inequality, using \eqref{eq:stopping}, to find
\begin{equation}
\label{eq:thm4-proof-triangle} 
\| \vhat(\bsalpha) - \bv \| 
	\leq \tau + \delta + \gamma \| \bsalpha \|.
\end{equation}
To bound $\bsalpha$, let
$\vtilde := \vhatg(\bsalpha) = \vhat(\bsalpha) + \gamma \bsalpha$.
Then
\begin{align}
\label{eq:vtilde_bound}
 \| \vtilde - \bv \| \le \| \vtilde - \bv^\delta \| + \| \bv^\delta - \bv \|
  \stackrel{\eqref{eq:stopping}}{\le} \tau + \delta
\end{align}
and, since
$\bsalpha = \alphahatg(\vtilde)$, 
\begin{align}
 \bsalpha = \alphahatgv +
  \int_0^1 (H(\alphahatg(\bv + s (\vtilde - \bv))) + \gamma I)^{-1}
  (\vtilde - \bv) \intds.
\end{align}
Thus $\|\bsalpha\|$ is bounded by
\begin{align}\label{eq:alphagdt-bnd-g}
 \| \bsalpha \| \le
  \| \alphahatgv \| + \frac1{\gamma} \| \vtilde - \bv \|
  \stackrel{\eqref{eq:vtilde_bound}}{\le} \| \alphahatgv \|
  + \frac{\tau+\delta}{\gamma}.
\end{align}
The term $\| \alphahatgv \|$ can be further bounded because $\|\alphahatgv\|$ 
is a decreasing function of $\gamma$:
\begin{subequations}\label{eq:mult-dec-with-gam}
\begin{align}
 \frac{\partial}{\partial \gamma}\|\alphahatgv\|^2
  &= 2 \alphahatgv \cdot \frac{\partial}{\partial \gamma} \alphahatgv \\
  &= 2 \alphahatgv \cdot \left( - \left.
   \frac{\partial \vhatg}{\partial \bsalpha} \right|_{\alphahatgv}
   \left. \frac{\partial \vhatg}{\partial \gamma} \right|_{\alphahatgv}
   \right) \\
  &= -2 \alphahatgv \cdot ((H_\gamma^{-1}(\alphahatgv) \alphahatgv) \\
  &\le 0,
\end{align}
\end{subequations}
The derivative of $\alphahatg$ with respect to $\gamma$ is computed by 
differentiating both sides of $\vhatg(\alphahatgv) = \bv$ with respect to 
$\gamma$, as in the implicit function theorem.
Since the continuity of $\alphahatg$ with respect to $\gamma$ at
$\gamma = 0$ is a consequence of the same continuity of $\vhatg$, we can extend 
\eqref{eq:alphagdt-bnd-g} to
\begin{align}\label{eq:alphagdt-bnd}
 \| \bsalpha \| \le \| \alphahatv \| + \frac{\tau+\delta}{\gamma}
  \le M + \frac{\tau+\delta}{\gamma}.
\end{align}
Setting the bound \eqref{eq:alphagdt-bnd} into \eqref{eq:thm4-proof-triangle}
yields \eqref{eq:proj-error-tau}
\end{proof}

If $\gamma \le C\delta$ and $\tau \le C'\delta$, then the error 
$\|\vhat(\bsalpha) - \bv\|$ of the approximate projection is $\cO(\delta)$.
Thus we achieve a bound like that of \thmref{acc} but with 
constants independent of the specific moment vectors $\bv$ and $\bv^\delta$, as
long as $\bv \in \cR^M$.

We now turn to the relationship between the stopping criterion 
\eqref{eq:stopping} and that of
\cite{engl1989convergence,engl1993convergence}.
In the latter, a distribution $g$ is called \emph{$\tau'$-optimal} if,
for a given tolerance $\tau' \in (0, \infty)$, it satisfies
\begin{align}\label{eq:tau-opt-tik}
 \cH(g) + \frac1{2\gamma} \left\| \Vint{\bm g} - \bv \right\|^2
  \le \hgv + \tau',
\end{align}
where $\hgv$ is the infimum of $\cH_\gamma(\cdot; \bv)$; see
\eqref{eq:hg}.
Because $\hgv$ is typically unknown, we cannot practically enforce 
\eqref{eq:tau-opt-tik} as is.
However, we find a computable and stronger criterion by considering the duality 
gap \cite[\S 5.5.1]{boyd2004convex}.
Indeed, for any $\bsalpha$ we have
\eqref{eq:reg-dual-grad}
\begin{equation}
  \bsalpha \cdot \bv - \Vint{\eta_*(\bsalpha \cdot \bm)}
   - \frac{\gamma}2 \| \bsalpha \|^2
  \le \hgv;
\end{equation}
so if the multiplier vector $\bsalpha$ further satisfies
\begin{equation}
\label{stopping-criteria-inequality}
\cH(\Ga)
 + \frac1{2\gamma} \| \Vint{\bm \Ga} - \bv \|^2
  \le \bsalpha \cdot \bv - \Vint{\eta_*(\bsalpha \cdot \bm)}
   - \frac\gamma2 \| \bsalpha \|^2 + \tau',
\end{equation}
we can conclude that it satisfies \eqref{eq:tau-opt-tik}.
Since the optimal duality gap is zero, \eqref{stopping-criteria-inequality} can 
be achieved for any $\tau' > 0$.

Now, the form of $\Ga$ and the fact that $\eta$ and $\etad$ are Legendre duals 
imply that  
\begin{align}
\label{eq:stopping-criteria-legendre}
 \eta(\Ga) + \eta_*(\bsalpha \cdot \bm)
  &= \bsalpha \cdot \bm \Ga.
\end{align}
This relation reduces \eqref{stopping-criteria-inequality} to
\begin{subequations}
\begin{align}
  \frac1{2\gamma} \| \Vint{\bm \Ga} - \bv
   + \gamma\bsalpha \|^2 \le \tau'.
\end{align}
\end{subequations}
which gives a stopping criterion equivalent to \eqref{eq:stopping}, where the 
tolerances are related by $\tau = \sqrt{2\gamma\tau'}$.

\subsection{Accuracy tests}\label{sec:static-tests}

To verify accuracy numerically, we consider the following curve in the
realizable set:
\begin{equation}\label{eq:static-moment-curve}
\bu(x) := \Vint{\bm \exp(\alpha_0(x) + \alpha_1(x) \mu)},
\end{equation}
where $x \in [-\pi, \pi]$ and
\begin{subequations}
	\begin{align}
	\alpha_0(x) := - \sin(x) + c \quand 
	\alpha_1(x) :=  K + \sin(x).
	\end{align}
\end{subequations}
The parameter $K$ is used to move the moment curve closer to the boundary of 
the realizable set.
The constant $c$ is set to
\begin{align}
c &= \log\left( \cfrac{K - 1}{2\sinh(K - 1)} \right) - 1
\end{align}
so that $1 = \max_x u_0(x)$.

We generate an error-contaminated moment vector $\bv^\delta$ 
by projecting the moment curve on the interval $[0, \dx]$ onto the space of 
polynomials up to degree $k - 1$.
We let $\bu_{\dx}(0)$ denote the evaluation at $x = 0$ of the $(k - 1)$-th 
degree polynomial projection of $\bu(x)$ given in 
\eqref{eq:static-moment-curve} on $[0, \dx]$.
We choose the edge $x = 0$ simply because it would appear in a finite-volume 
method.

The velocity space is $V = [-1, 1]$ (as in the numerical tests in 
\secref{numerics} below), and for the basis functions $\bm$ we take the 
Legendre polynomials up to seventh order.
We compute the velocity integrals in \eqref{eq:static-moment-curve} using a 
forty-point Gauss--Lobatto quadrature and the spatial inner products for the 
orthogonal projection with a twenty-point Gauss--Lobatto quadrature.

In \tabref{static}, we plot the error
\begin{align}
\| \vhat(\bsalpha_\gamma^\tau(\bu_{\dx}(0))) - \bu(0) \|,
\end{align}
where $\bsalpha_\gamma^\tau(\bu_{\dx}(0)))$ denotes the first Newton iterate 
satisfying the stopping criterion \eqref{eq:stopping} for the moment vector
$\bu_{\dx}(0)$.
The test results confirm that the appropriate choices of $\gamma$ and $\tau$ 
give the expected orders of convergence.
We note that almost all of the moment vectors $\bu_{\dx}(0)$ generated by the 
polynomial projections for the table are not realizable.

\begin{table}
	\small
	\centering
	\begin{tabular}{rrrrrrr}
		
		& \multicolumn{2}{c}{$k = 2$}
		& \multicolumn{2}{c}{$k = 3$}
		& \multicolumn{2}{c}{$k = 4$} \\
		
		\cmidrule(r){2-3} \cmidrule(r){4-5}   \cmidrule(r){6-7}
		
		$\dx$ & $E^1_h$ & $\nu$
		& $E^1_h$ & $\nu$
		& $E^1_h$ & $\nu$ \\ \midrule
		
$2^{-1}$ & 5.7e-01 &       & 4.2e-01 &       & 2.8e-01 &      \\
$2^{-2}$ & 2.8e-01 & 0.99  & 6.9e-02 & 2.62  & 1.5e-02 & 4.23 \\
$2^{-3}$ & 6.9e-02 & 2.05  & 1.0e-02 & 2.77  & 1.3e-03 & 3.55 \\
$2^{-4}$ & 1.5e-02 & 2.19  & 1.3e-03 & 2.97  & 7.5e-05 & 4.10 \\
$2^{-5}$ & 4.5e-03 & 1.75  & 1.4e-04 & 3.18  & 4.8e-06 & 3.95 \\
$2^{-6}$ & 1.3e-03 & 1.81  & 2.0e-05 & 2.83  & 3.1e-07 & 3.97 \\
$2^{-7}$ & 3.2e-04 & 1.99  & 2.5e-06 & 3.02  & 1.9e-08 & 4.03 \\
$2^{-8}$ & 7.5e-05 & 2.11  & 3.1e-07 & 3.00  & 1.1e-09 & 4.05 \\
$2^{-9}$ & 2.0e-05 & 1.90  & 3.8e-08 & 3.03  & 7.3e-11 & 3.97
		
	\end{tabular}
	\caption{Regularization errors for $K = 200$ and $\gamma = \tau = 
\dx^k$.}
	\label{tab:static}
\end{table}

%% file: sections/numerics.tex
\section{Numerical results}
\label{sec:numerics}

In this section we demonstrate that \eqref{eq:reg-mn} can be simulated
using an off-the-shelf, high-order method for hyperbolic 
conservation laws.
Our simulations indicate that the results of \secref{acc} can be used to guide 
the choice of the regularization parameter $\gamma$ and the optimization 
tolerance $\tau$ so that a numerical solution of \eqref{eq:reg-mn} is an 
accurate solution of the original entropy-based moment system \eqref{eq:mn}.
We also present numerical simulations of a benchmark problem.

For numerical tests, we consider a kinetic equation that describes particles of 
unit speed moving through a material with slab geometry (see e.g., 
\cite{Lewis-Miller-1984}):
\begin{align}\label{eq:slab}
 \partial_t f + \mu \partial_x f + \sig{a} f = \sig{s}(\Vint{f} - f) + S.
\end{align}
The spatial domain is $X = (\xL, \xR)$ is one-dimensional, and the velocity 
variable $\mu \in [-1, 1]$ gives the cosine of the angle between the 
microscopic velocity and and the $x$-axis.
The collision operator here is $\cC(f) := \sig{s}(\vint{f} - f)$, where
$\sig{s} \ge 0$ is the \emph{scattering cross section}.
This collision operator $\cC$ represents isotropic scattering, is linear, and 
dissipates any convex entropy $\eta$.
Our equation also has a loss term $\sig{a} f$, where $\sig{a} \ge 0$ is  
the \emph{absorption cross section}, as well as a source $S = S(t, x, \mu)$.
Equation \eqref{eq:slab} is supplemented with the initial conditions
\begin{align}
 f(0, x, \mu) = f_0(x, \mu)
\end{align}
and boundary conditions
\begin{equation}
 f(t, \xL, \mu >0) = f_{\rm L}(t, \mu) \qquand 
f(t, \xR, \mu < 0) = f_{\rm R}(t, \mu) .
\end{equation}

The original entropy-based moment equations for \eqref{eq:slab} are
\begin{align}\label{eq:mn-slab}
 \partial_t \bu + \partial_x \bff(\bu) + \sig{a} \bu
  = \sig{s}R \bu + \bs,
\end{align}
where $R = \diag\{0, -1, \dots , -1\}$ and $\bs := \vint{\bm S}$.
The regularized entropy-based moment equations for \eqref{eq:slab} are
\begin{align}\label{eq:rmn-slab}
 \partial_t \bu + \partial_x \bff_\gamma(\bu) + \sig{a} \bu
  = \sig{s}R \vhat(\alphahatgu) + \bs.
\end{align}

\begin{remark}
To achieve the entropy-dissipation property described in 
\secref{rmn-structure}, 
we must use the regularized moment vector $\vhat(\alphahatgu)$ in the collision 
operator.
This makes the collision operator nonlinear.
The absorption term is not part of the collision operator and thus not part of 
the entropy-dissipating structure of the kinetic equation \eqref{eq:slab}, and 
for that reason we simply leave it as a linear decay term in the regularized 
moment equations.
\end{remark}

The initial conditions are $\bu(0, x) = \Vint{\bm f_0(x, \cdot)}$ and
and to define the boundary conditions we extend the definitions of $f_{\rm L}$ 
and $f_{\rm R}$ from $\mu \in [0, 1]$ and $\mu \in [-1, 0]$ respectively to all 
$\mu \in [-1, 1]$ to get
\begin{align}
 \bu(t, \xL) = \bu_{\rm L}(t) := \Vint{\bm f_{\rm L}(t, \cdot)}
  \quand
 \bu(t, \xR) = \bu_{\rm R}(t) := \Vint{\bm f_{\rm R}(t, \cdot)}.
\end{align}
While this is not technically correct (and proper treatment of boundary 
conditions for moment methods remains an open problem), we only consider 
problems where the boundary conditions have at most a negligible effect on the 
solution.

For our numerical tests we take the Maxwell--Boltzmann entropy because it has 
generic physical relevance and leads to a nonnegative entropy ansatz.
We take the Legendre polynomials up to order $N$ for the basis functions in 
$\bm$, and so the number of moment components is $n = N + 1$.

\subsection{Numerical method}
\label{sec:dg}

Two common high-order methods for hyperbolic equations are the 
discontinuous-Galerkin (DG) \cite{CKS2000,CockburnShuIII} and
weighted-essentially-nonoscillatory (WENO) \cite{Shu1998} methods.
The main consideration in selecting a method is the number of times one must 
compute multipliers $\alphahatgu$ via \eqref{eq:reg-mult}, since this is the 
most expensive part of the algorithm.
For the hyperbolic component, WENO  offers a more attractive choice, since it 
only requires multipliers at the cell edges (and cell means, if a 
characteristic transformation is performed) in order to compute fluxes.
A DG algorithm, on the other hand, needs multipliers on a 
quadrature set in each cell in order to evaluate the volume term.  
However, for the regularized equations \eqref{eq:rmn-slab} the collision term 
is 
nonlinear, and thus both the WENO and DG methods must approximately integrate 
this term in space with a quadrature, and each quadrature evaluation requires
knowledge of the multipliers. Thus the advantages of WENO over DG are lost.  We 
therefore proceed with a DG implementation, a description of which (in the 
context of solving \eqref{eq:mn-slab}) can be found in 
\cite{AlldredgeSchneider2014}.
The implementation here is essentially the same, except that a realizability 
limiter is not needed.

As in \cite{AlldredgeSchneider2014} we use the Lax-Friedrichs numerical flux.
The numerical dissipation constant is set to one because the eigenvalues of the 
flux $\bff_\gamma$ have the bound
\begin{align}
 \lambda_{\max}\left( \frac{\partial \bff_\gamma}{\partial \bv} \right) \le 1,
\end{align}
where $\lambda_{\max}$ denotes the maximum (in absolute value) eigenvalue.
This is a straightforward extension of \cite[Lemma 3.1]{AlldredgeSchneider2014}.
We use SSP Runge--Kutta methods for time integration, specifically those given 
in \cite{ketcheson2008highly}:
for the second-order results, we use the $s$-stage method with ten stages;
for the third-order results, we use the $r^2$-stage method with $r = 4$; and
for the fourth-order results, we use the ten-stage method.
We use a regular grid with $N_x$ spatial cells with width
$\dx = (\xR - \xL) / N_x$.
The DG basis consists of polynomials up to degree $k - 1$ on each cell.
We choose the time step as in \cite{AlldredgeSchneider2014}:
\begin{align}
 \dt = \frac{w_Q \dx}{1 + w_Q \dx (\sig{a} + \sig{s})},
\end{align}
where $w_Q$ is the weight of the endpoints of the $Q$-point Gauss-Lobatto 
quadrature with $2Q - 2 \ge k$.
While in \cite{AlldredgeSchneider2014} this time step was chosen in order to 
maintain realizability of the cell means, which is irrelevant to us, we found 
that trying to use smaller time steps quickly led to stability problems.

We solve dual optimization problem \eqref{eq:reg-mult} using a
Levenberg-Marquardt-type algorithm and the Armijo line search.

\subsection{Convergence test using a manufactured solution}

To test how accurately the numerical solution of \eqref{eq:rmn-slab} 
approximates the numerical solution of the original entropy-based moment
equations \eqref{eq:mn-slab}, we used the method of manufactured solutions, in 
particular the one proposed in \cite{AlldredgeSchneider2014}:
Let
\begin{align}\label{eq:man-soln-exact}
 \bw(t, x) :=  \Vint{\bm \exp(\alpha_0(t, x) + \alpha_1(t, x) \mu)},
\end{align}
where
\begin{align}
 \alpha_0(t, x) := - \sin(x - t) + 4 t + c
 \qquand 
 \alpha_1(t, x) :=  K + \sin(x - t).
\end{align}
As above, the parameter $K$ is used to move the moment curve closer to the 
boundary of the realizable set, and the constant $c$ is set to
\begin{align}
 c &= \log\left( \cfrac{K - 1}{2\sinh(K - 1)} \right) - 1 - 4\tf,
\end{align}
so that $1 = \max_{t, x} w_0(t, x)$.
The spatial domain is $X = (-\pi, \pi)$, and we take the final time
$\tf := \pi / 5$.
We use periodic boundary conditions and include neither scattering nor 
absorption: $\sig{a} = \sig{s} = 0$. Since the goal is to converge to the 
solution of the original entropy-based moment equations \eqref{eq:mn-slab}, we
compute the source $s$ for the manufactured solution using $\bff$ instead of
$\bff_\gamma$; that is, we set ${\bs = \partial_t \bw + \partial_x \bff(\bw)}$.

Error is measured in the $L^1$ sense:
Let $\bu_{\dx}(\tf, x)$ denote the point-wise evaluation of the DG solution at 
the final time; we consider the errors only in the zeroth component, which are
given by
\begin{align}
 e_{\dx} = \int_{\xL}^{\xR} |u_{\dx,0}(\tf, x) - w_0(\tf, x)| \intdx.
\end{align}
We approximate the integral with a twenty-point Gauss--Lobatto quadrature in 
each spatial cell.  Results are given in \tabref{manufactured}; 
we used a factor $10^{-1}$ in front of the $\dx^k$ for $\gamma$ and $\tau$ 
(unlike in \secref{static-tests}).
With this factor, we see the expected orders of convergence for different 
values of $k$.
For larger values of this factor, the observed convergence in our tests is 
slightly smaller than expected.

\begin{table}
\footnotesize
\centering
\begin{tabular}{rcrcrcr}

 & \multicolumn{2}{c}{$k = 2$}
 & \multicolumn{2}{c}{$k = 3$}
 & \multicolumn{2}{c}{$k = 4$} \\
 
\cmidrule(r){2-3} \cmidrule(r){4-5}   \cmidrule(r){6-7}

$N_x$ & $e_{\dx}$ & $\nu$ & $e_{\dx}$ & $\nu$ & $e_{\dx}$ & $\nu$ \\ \midrule

  10 & 1.7184e-01 &   --  & 6.9233e-02 &   --  & 9.3886e-03 &   --  \\
  20 & 1.1080e-01 &  0.63 & 6.6889e-03 &  3.37 & 2.9149e-04 &  5.01 \\
  40 & 2.9046e-02 &  1.93 & 1.2543e-03 &  2.41 & 4.6188e-05 &  2.66 \\
  80 & 7.6273e-03 &  1.93 & 1.7001e-04 &  2.88 & 5.9739e-06 &  2.95 \\
 160 & 2.2065e-03 &  1.79 & 2.3744e-05 &  2.84 & 4.7288e-07 &  3.66 \\
 320 & 5.5530e-04 &  1.99 & 3.0206e-06 &  2.97 & 3.0056e-08 &  3.98 \\
 640 & 1.4088e-04 &  1.98 & 3.8838e-07 &  2.96 & 1.9219e-09 &  3.97 \\
1280 & 3.6005e-05 &  1.97 & 4.8748e-08 &  2.99 & 1.1919e-10 &  4.01

\end{tabular}
\caption{Errors between numerical solutions of the regularized equations to 
the exact solution of the original equations for the manufactured-solution test.
Here, $N = 3$, $K = 5$, and we use the regularization and optimization 
parameters $\gamma = \tau = 10^{-1}\dx^k$.}
\label{tab:manufactured}
\end{table}

\subsection{Plane-source benchmark}
\label{sec:plane}

The plane-source problem \cite{ganapol1977generation} tests how well a method
handles strong spatial gradients and angular distributions with highly 
localized support.
We use a slightly smoothed version of this problem in which an initial delta
function in space is replaced by a narrow Gaussian.
Even with this smoothing, solutions are rough and numerical 
convergence is slow.

The domain is $X = (-1.2, 1.2)$, and the initial conditions are given by
\begin{align}
 f_{t = 0}(x, \mu) = \max\left(\frac{\exp\left(-x^2 / \Sigma^2 \right)} 
  {\Sigma}, f_{\rm floor}\right),
\end{align}
where $\Sigma = 0.01$, and $f_{\rm floor} = 0.5 \times 10^{-8}$ approximates a
vacuum.
(The ansatz with the Maxwell--Boltzmann entropy, which has the form 
$\exp(\bsalpha \cdot \bm)$, cannot be exactly zero.)
The boundary conditions
$f_{\rm L}(t, \mu) \equiv f_{\rm R}(t, \mu) \equiv f_{\rm floor}$ are 
consistent with the analytical solution.
We simulate the solution up to $\tf = 1$.

For the results in this section, we first found the smallest values of $\gamma$ 
and $\tau$ with which we could reliably compute numerical solutions of the 
regularized equations without the optimizer crashing.
These values were $\gamma = 10^{-6}$ and $\tau = 10^{-7}$.
Then with these values, we compute a very accurate, nearly converged 
numerical solution using the fourth-order DG method with 4000 spatial 
cells.
We compare this solution with a high-resolution solution of the original 
entropy-based moment equations, which we generate using the second-order 
kinetic scheme of \cite{AllHau12}.
We use 13000 cells with the kinetic scheme and even for the slightly 
smoothed version of the problem considered here, we do have to use the 
technique of isotropic regularization for some moment vectors in the numerical 
solution (see \cite{AllHau12} for details).

\figref{plane-kin-comp} shows the results for $N = 5$.
In this figure, the solutions are indistinguishable, but in 
Figures \ref{fig:plane-kin-comp-zoom-1} and \ref{fig:plane-kin-comp-zoom-2}, we 
zoom in on the solutions in two places to show that differences on the order of 
0.01, or about 1\%, remain.
We computed solutions for other values of $N$ and found similar results.

To get some understanding of the effect of the value of $\gamma$ on 
the solutions, we also present numerical solutions to the plane-source problem 
with two larger values of $\gamma$.
(We continue to use $\tau = 10^{-7}$ and the fourth-order DG method with 4000 
cells.)
In \figref{plane-gam-comp} we include the results using $\gamma = 10^{-2}$.
While for the larger value of $\gamma$ the first and third waves of the 
solution are larger in magnitude, the second wave is smaller and somewhat 
delayed.
The front of the third wave is also slightly delayed.
It seems to us that while increasing the value of $\gamma$ does not have a 
smoothing effect, it does seem to have a delaying effect like that predicted by 
the analysis of the regularized P$_N$ equations in \secref{reg-pn}.
The zoomed-in plots in Figures \ref{fig:plane-gam-comp-zoom-1} and
\ref{fig:plane-gam-comp-zoom-2} include a third, intermediate value of $\gamma$ 
which confirms this observation.

\begin{figure}
 \centering
 \begin{subfigure}[t]{0.45\textwidth}
  \includegraphics[width=\textwidth]{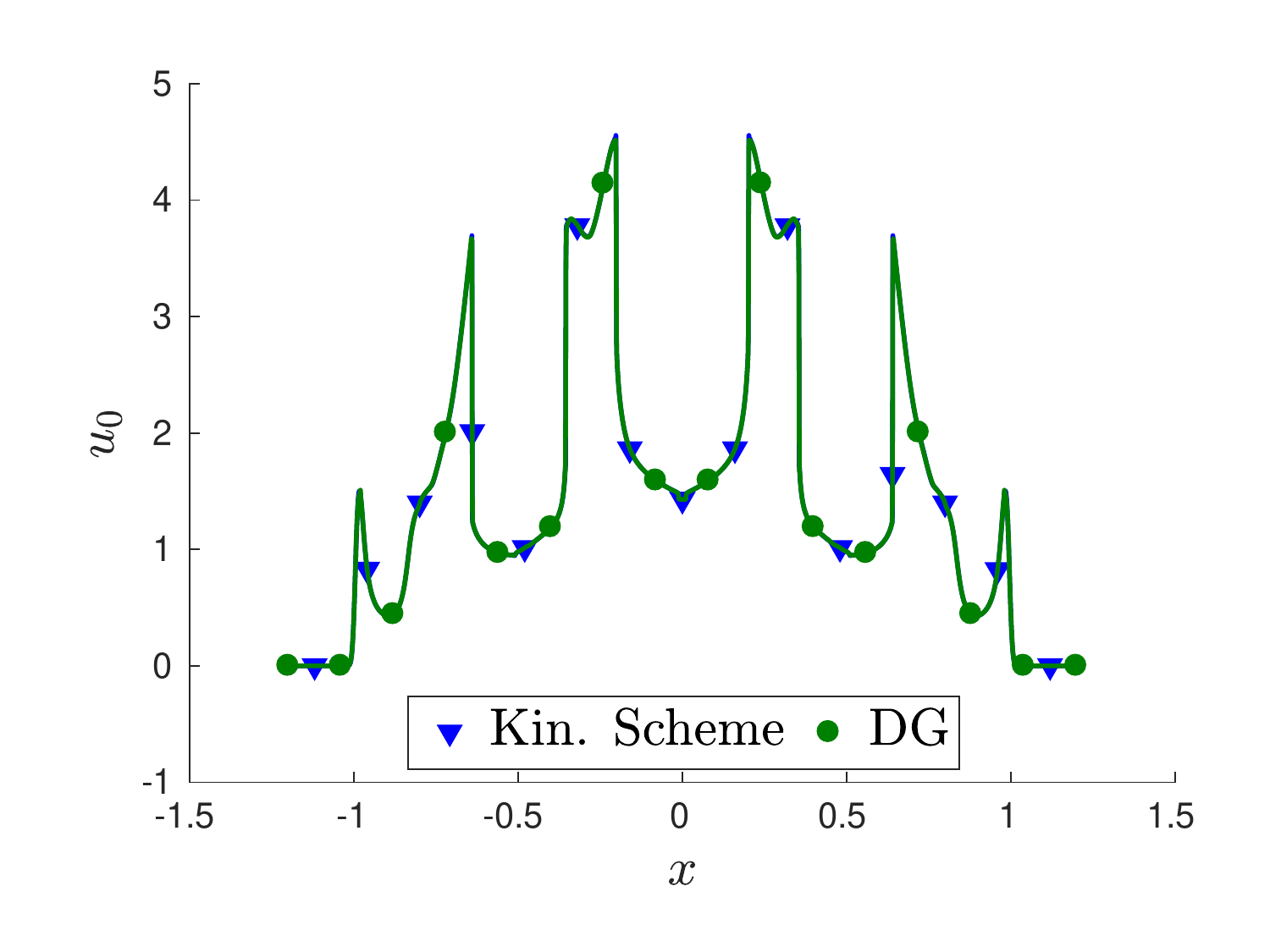}
  \caption{Plane-source solution using the kinetic scheme and the fourth-order 
   DG scheme with $\gamma = 10^{-6}$.}
  \label{fig:plane-kin-comp}
 \end{subfigure}
 \quad
 \addtocounter{subfigure}{2}
 \begin{subfigure}[t]{0.45\textwidth}
  \includegraphics[width=\textwidth]{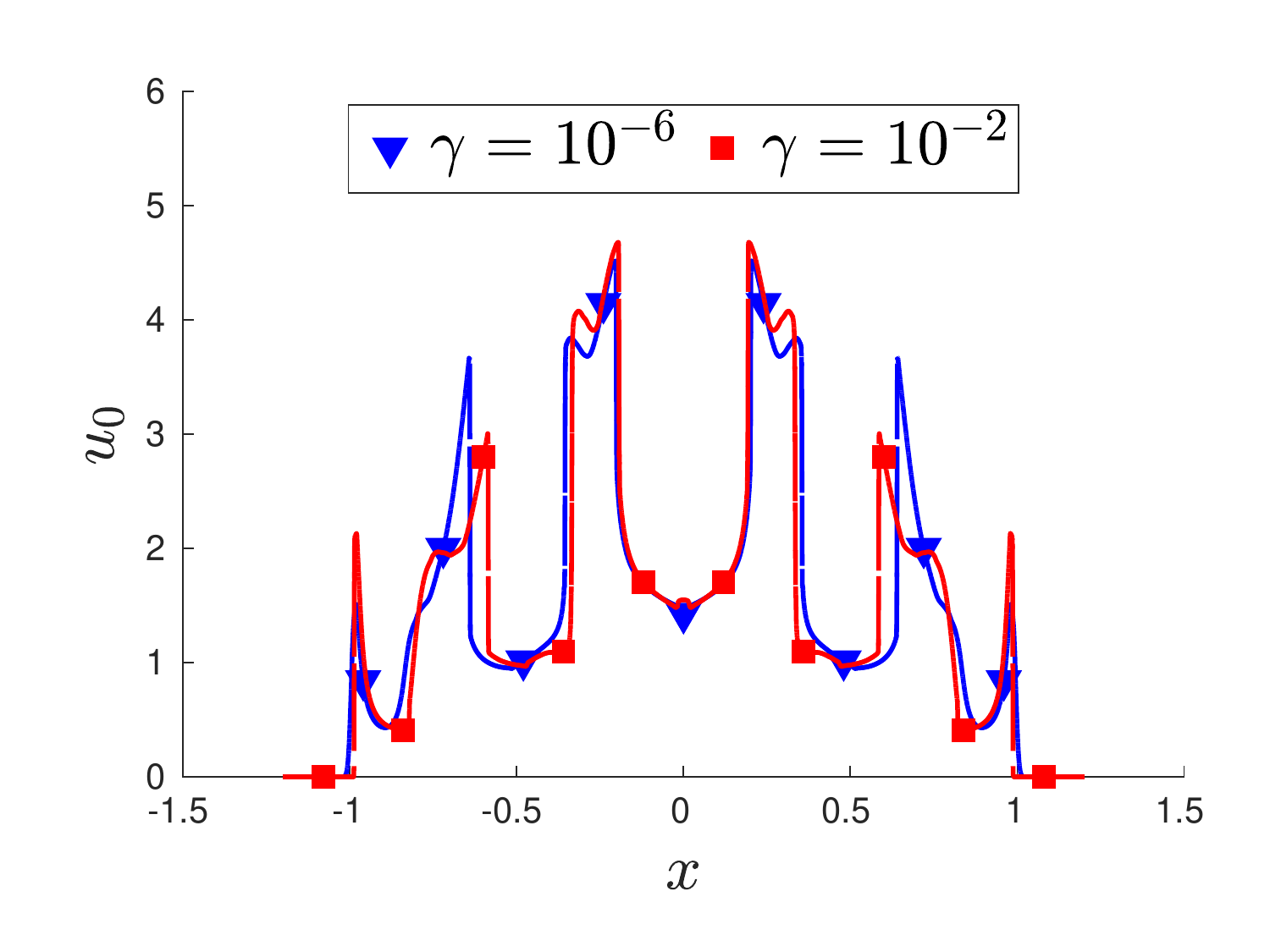}
  \caption{Plane-source solution using two different values of $\gamma$.}
  \label{fig:plane-gam-comp}
 \end{subfigure}
 \\
 \addtocounter{subfigure}{-3}
 \begin{subfigure}[t]{0.45\textwidth}
  \includegraphics[width=\textwidth]{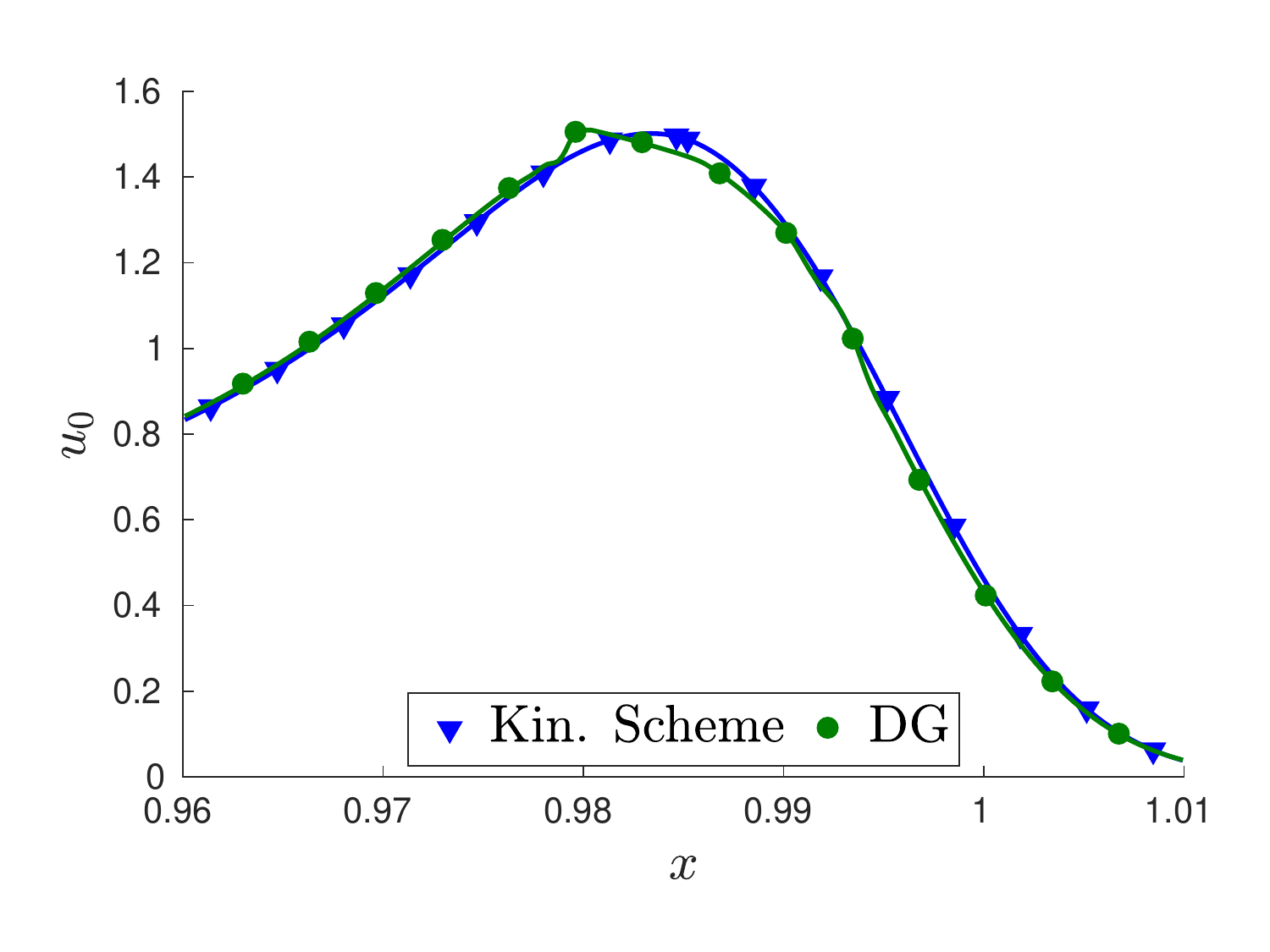}
  \caption{Zoom-in around $x \in [0.96, 1.01]$ of the comparison of the 
   solutions from the old kinetic scheme and the new regularized equations.}
  \label{fig:plane-kin-comp-zoom-1}
 \end{subfigure}
 \quad
 \addtocounter{subfigure}{2}
 \begin{subfigure}[t]{0.45\textwidth}
  \includegraphics[width=\textwidth]{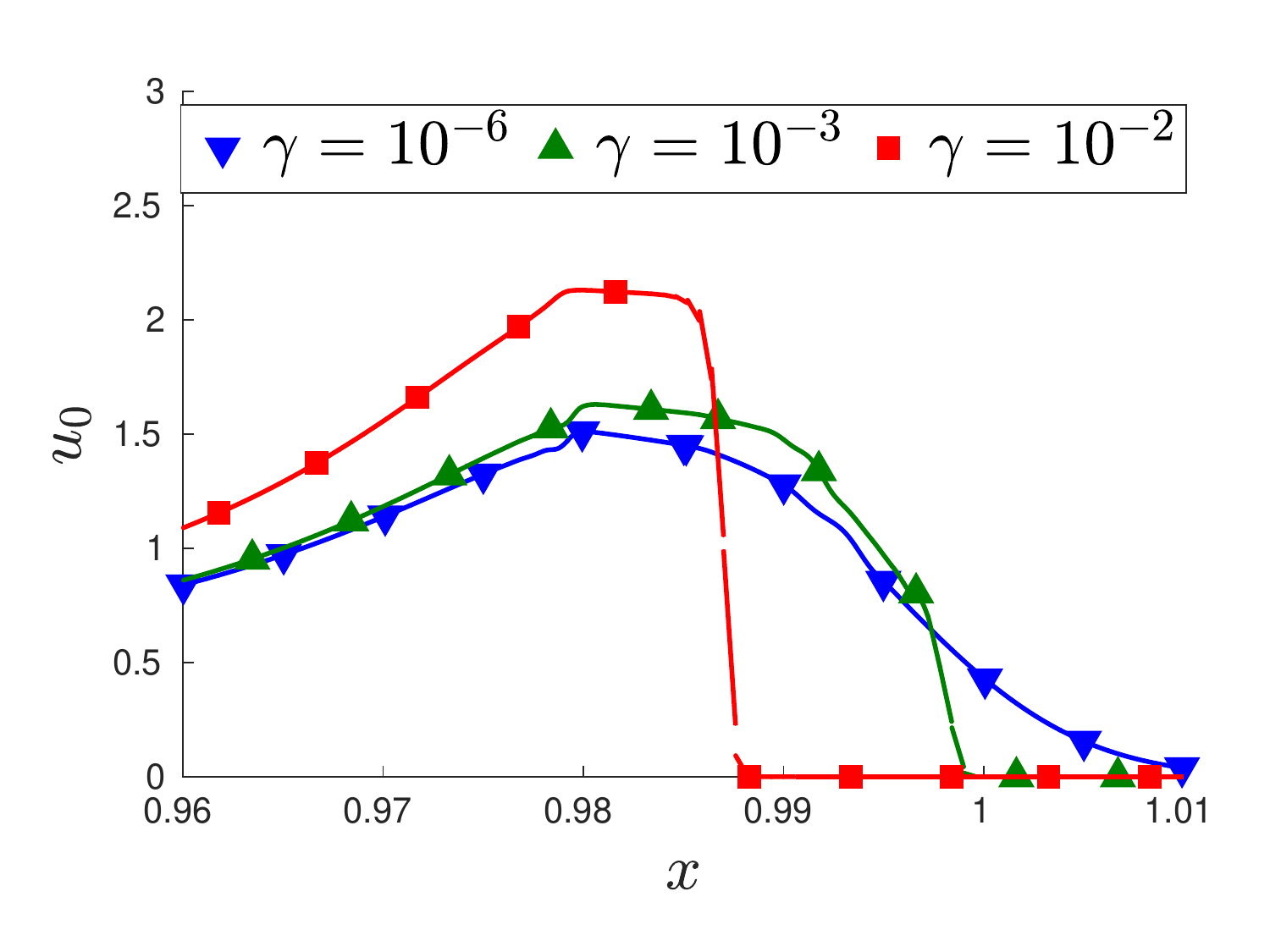}
  \caption{Zoom-in around $x \in [0.96, 1.01]$ of the comparison of the 
   solutions for different values of $\gamma$.}
  \label{fig:plane-gam-comp-zoom-1}
 \end{subfigure}
 \\
 \addtocounter{subfigure}{-3}
 \begin{subfigure}[t]{0.45\textwidth}
  \includegraphics[width=\textwidth]{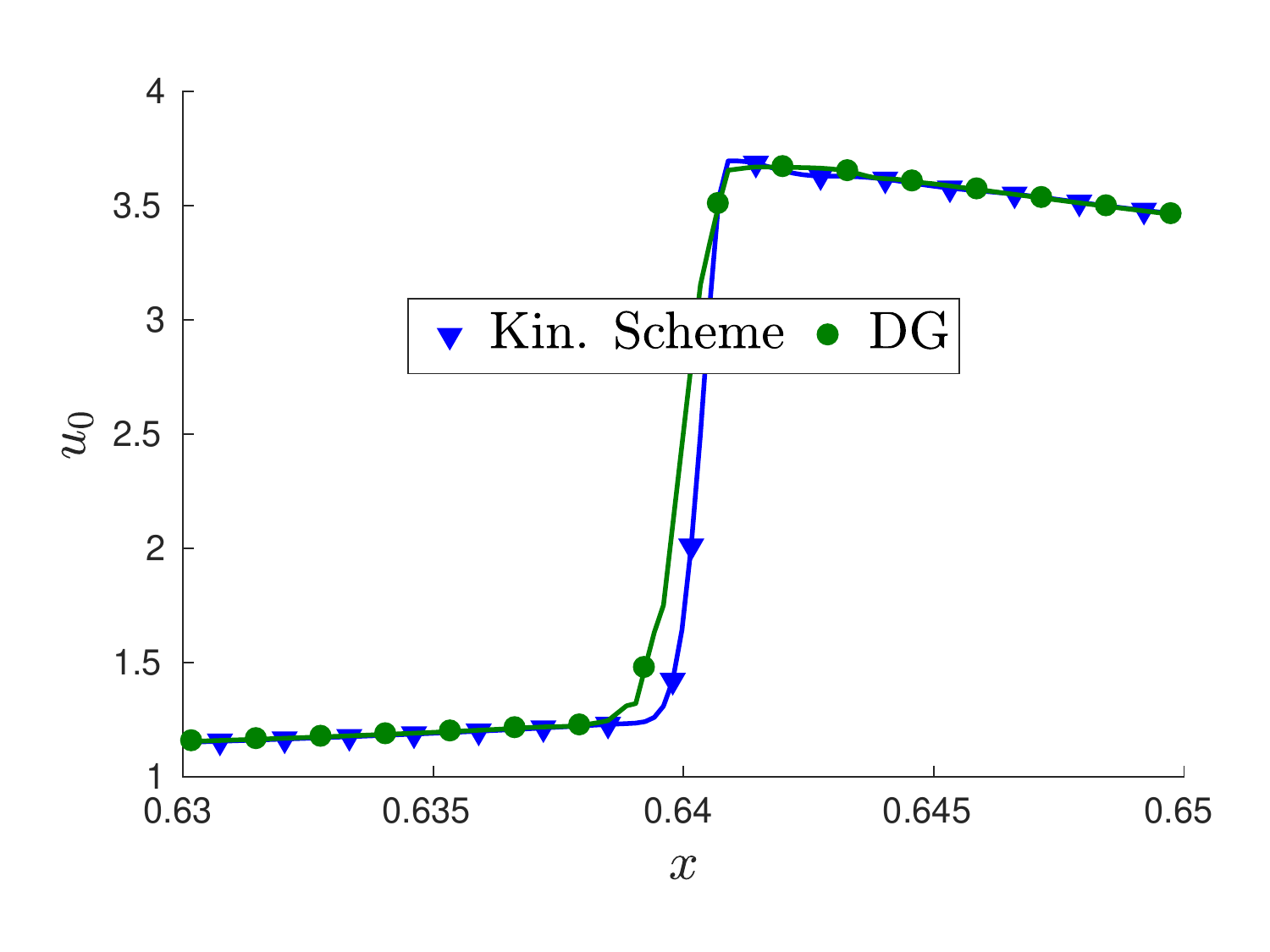}
  \caption{Zoom-in around $x \in [0.63, 0.65]$ of the comparison of the 
   solutions from the old kinetic scheme and the new regularized equations.}
  \label{fig:plane-kin-comp-zoom-2}
 \end{subfigure}
 \quad
 \addtocounter{subfigure}{2}
 \begin{subfigure}[t]{0.45\textwidth}
  \includegraphics[width=\textwidth]{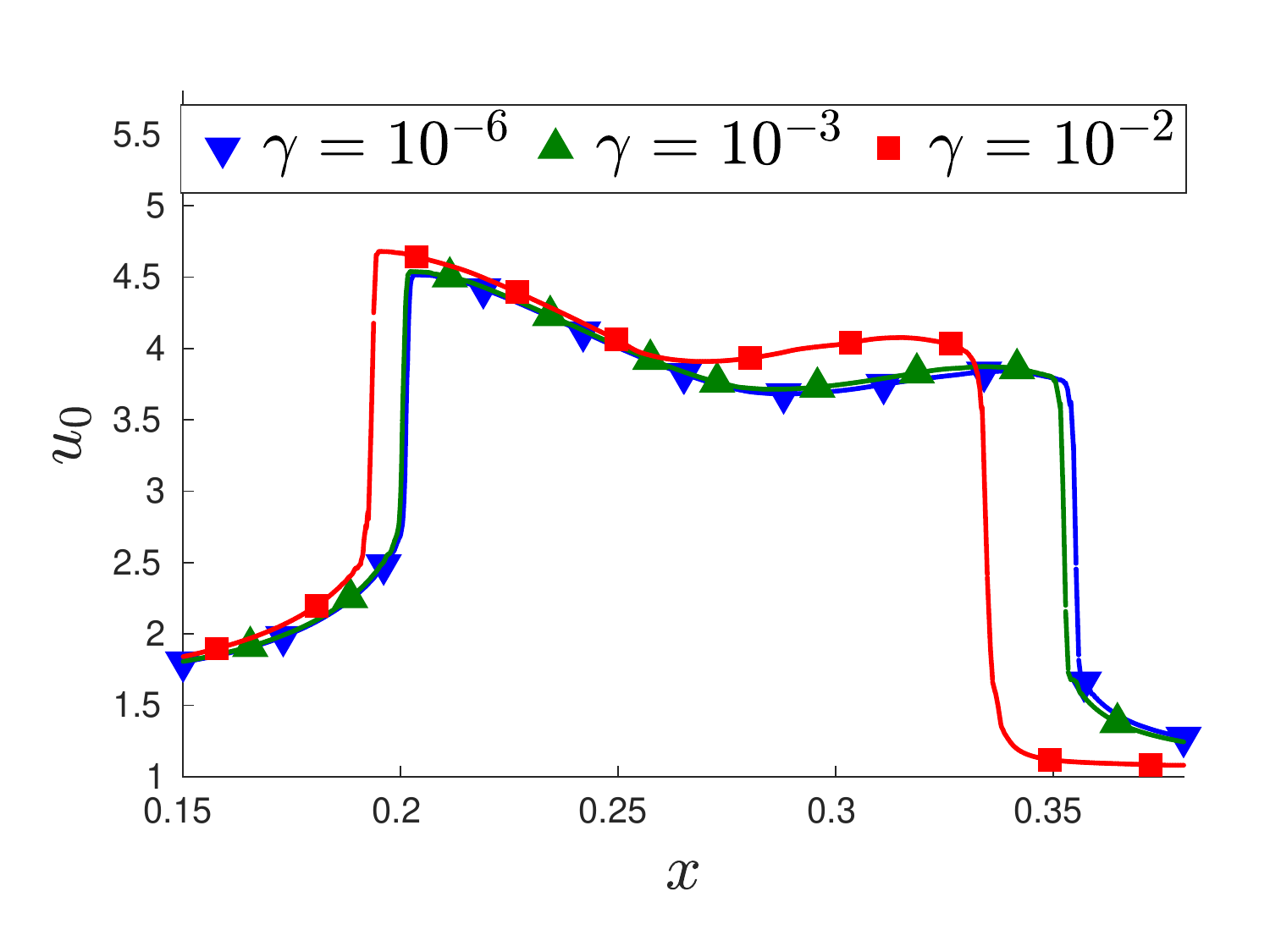}
  \caption{Zoom-in around $x \in [0.15, 0.4]$ of the comparison of the 
   solutions for different values of $\gamma$.}
  \label{fig:plane-gam-comp-zoom-2}
 \end{subfigure}
 \caption{Numerical solutions of the plane-source problem.}
 \label{fig:plane-all}
\end{figure}

%% file: sections/conclusions.tex
\section{Concluding remarks}

In this work we introduce a new moment method for kinetic equations.
We derive this method, dubbed the regularized entropy-based moment 
method, by starting with the original entropy-based moment 
equations and relaxing the equality constraint in the optimization defining the 
ansatz reconstruction for the flux and collision terms.
By relaxing these constraints, we can define flux and collision terms 
for nonrealizable moment vectors which, while unphysical, often appear as a 
result of discretization error in numerical simulations.
The relaxation corresponds to a Tikhonov regularization in the defining 
optimization problem's dual.

We showed that the regularized system keeps many of the same properties as 
the original system:
Firstly, it dissipates entropy, albeit not the same as the original 
system but an approximation thereof, and is hyperbolic.
When the basis functions are orthonormal, the regularized system is also
rotationally symmetric.
On the other hand, translational invariance is lost.
The problem of degenerate densities for unbounded velocity domains also carries 
over to the regularized problem in the form of moment 
vectors for which the regularized problem has no solution.

We view these regularized equations as a tool to compute approximate solutions 
to the original entropy-based moment equations because the error in the 
regularized reconstruction can be controlled through the choice of the 
regularization parameter.
Numerical simulations using a discontinuous-Galerkin (DG) scheme confirm this 
accuracy for the moment equations from a one-dimensional linear kinetic 
equation.
We can use the DG scheme essentially off-the-shelf because relaxing the 
realizability requirement greatly simplifies its implementation.

For possible future work, a rigorous proof of the accuracy of the regularized 
equations would put the accuracy results on more solid ground.
In one spatial dimension (where well-posedness theory for hyperbolic systems is 
available), perhaps the best route to this result is by examining the difference 
in the Jacobians of the fluxes $\bff_\gamma$ and $\bff$ and applying the results 
of \cite{bianchini2002stability}.

The scheme could be improved by using an adaptive choice of the regularization 
parameter.
Another improvement would be the development of an asymptotic-preserving scheme 
to handle stiff, collision-dominated kinetic regimes.
This has been long sought for entropy-based moment equations, and we believe 
this will be more easily attainable without the obstacle of realizability.

%% file: sections/junk.tex
\section{Degenerate densities}
\label{sec:junk}

We recall that in \propref{junk-line-g} we are considering the 
Maxwell--Boltzmann entropy and $V = \R^d$.
We let $\bm$ contain polynomials up to degree $N$, for some even $N$, such that 
the only component of degree greater than or equal to $N$ is
$m_{n - 1}(v) = |v|^N$.
Let $\cA$ be the set of multiplier vectors such that $\Ga \in L^1(V)$.
In particular, we know that
\begin{align}
 \cA \subset \{ \bsalpha \in \R^n : \alpha_{n - 1} \le 0 \}
  \qquand
 \cA \cap \partial \cA \subset \{ \bsalpha \in \R^n : \alpha_{n - 1} = 0 \}.
\end{align}

In order to prove \propref{junk-line-g}, we need the following lemma.

\begin{lemma}\label{lem:min-vhatg}
The function $\cH_\gamma(\cdot ; \bv)$ achieves a unique minimum if and only if
$\bv \in \vhatgA$.
The minimizer, if it exists, has the form 
$\Ga$, defined in \eqref{eq:ansatz}.
\end{lemma}
 
\begin{proof}
First assume $\bv = \vhatga$ for some $\bsalpha \in \cA$.
Since $\eta$ is convex, 
\begin{align}
\label{eq:convex-expansion-Ga}
 \Vint{\eta(g)} \ge \Vint{\eta(\Ga)}
  + \Vint{\bsalpha \cdot \bm (g - \Ga)},
\end{align}
where we use the fact that $\eta'(\Ga) = \bsalpha \cdot \bm$.
Applying \eqref{eq:convex-expansion-Ga} to the definition of $\cH_\gamma$
in \eqref{eq:H-gamma} and using the fact that $\bv = \Vint{\bm \Ga} + \gamma \bsalpha$ gives
\begin{subequations}
\begin{align}
 \cH_\gamma(g; \bv) &\ge \cH_\gamma(\Ga; \bv)
   + \frac1{2\gamma} \| \vint{\bm g} - \bv \|^2
   - \frac1{2\gamma} \| \vint{\bm \Ga} - \bv \|^2 \nonumber \\
  &\qquad + \Vint{\bsalpha \cdot \bm (g - \Ga)} \\
  &= \cH_\gamma(\Ga; \bv)
   + \frac1{2\gamma} \| \vint{\bm g} - \Vint{\bm \Ga}
    - \gamma \bsalpha \|^2
    - \frac\gamma 2 \| \bsalpha \|^2 \nonumber \\
  &\qquad + \Vint{\bsalpha \cdot \bm (g - \Ga)} \\
  &= \cH_\gamma(\Ga; \bv)
   + \frac1{2\gamma} \| \vint{\bm g} - \Vint{\bm \Ga}\|^2 \\
  &\ge \cH_\gamma(\Ga; \bv).
\end{align}
\end{subequations}
Thus $\Ga$ minimizes $\cH_\gamma(\cdot ; \bv)$.

On the other hand, assume $\cH_\gamma(\cdot ; \bv)$ has a minimizer, which 
we denote by $g^*$.
The minimizer $g^*$ also solves the problem
\begin{align}
\label{eq:g-star-solves}
 \minimize_{g \in \bbF(V)} \: \cH_\gamma(g; \bv)
  \qquad \st \: \Vint{\bm g} = \Vint{\bm g^*}.
\end{align}
Moreover, because the penalty term in $\cH_\gamma$ is constant on the 
constraint set in \eqref{eq:g-star-solves}, $g^*$ solves the original problem 
\eqref{eq:primal} with $\bv$ replaced by $\vint{\bm g^*}$.
Thus according to \cite[Theorem 9]{Hauck-Levermore-Tits-2008}, $g^* = \Ga$ for 
some $\bsalpha \in \cA$.

Now let $\widetilde g \in L^1(V)$ be smooth and have compact support.
We consider $\veps$ small enough such that
$\Ga + \veps \widetilde g \in \bbF(V)$.
Since $\Ga$ minimizes $\cH_\gamma(\cdot ; \bv)$, the Gateaux derivative of 
$\cH_\gamma$ in the direction $\widetilde g$ must be zero, i.e.,
\begin{subequations}
\begin{align}
 0 &= \lim_{\veps \to 0} \frac d {d\veps}
   \cH_\gamma(\Ga + \veps \widetilde g; \bv) \\
  &= \Vint{\eta'(\Ga) \widetilde g} + \frac1\gamma (\Vint{\bm \Ga} - \bv)
   \cdot \Vint{\bm \widetilde g} \\
  &= \left( \bsalpha + \frac1\gamma (\Vint{\bm \Ga} - \bv) \right)
   \cdot \Vint{\bm \widetilde g},
\end{align}
\end{subequations}
from which we can conclude, using the freedom in the choice of $\widetilde g$, 
that $\bv = \vhatga$.
\end{proof}

\begin{remark}
Like the original problem, when $V$ is compact there are no degenerate 
densities.
In this case, $\cA = \R^n$, and since the dual of the regularized problem is 
strongly convex, it has a maximizer for any $\bv \in \R^n$, and therefore 
$\vhatA = \R^n$.
\end{remark}

\begin{proof}[Proof of \propref{junk-line-g}]
Let 
\begin{align}\label{eq:dual-psi}
 \psi_\gamma(\bsalpha; \bv) := \bsalpha \cdot \bv
  - \Vint{\etad(\bsalpha \cdot \bm)}
  - \frac\gamma 2 \|\bsalpha\|^2
\end{align}
be the dual function for the regularized problem so that
\begin{align}
 \alphahatgv = \argmax_{\bsalpha \in \cA} \psi_\gamma(\bsalpha; \bv).
\end{align}
One can extend the arguments from \cite{Hauck-Levermore-Tits-2008} to show that 
$\alphahatgv$ exists for \emph{any} $\bv \in \R^n$ (although when
$\alphahatgv$ is on the boundary of $\cA$, it may not satisfy the first-order 
necessary conditions, i.e., sometimes $\vhatg(\alphahatgv) \ne \bv$).

Suppose that \eqref{eq:tik-primal} has a minimizer $g^*$.
Then by \lemref{min-vhatg}, $g^* = G_{\bsalpha^*}$ with
$\bv = \vhatg(\bsalpha^*)$.
Since the latter shows that $\bsalpha^*$ satisfies the first-order necessary 
conditions for \eqref{eq:dual-psi} and $\psi_\gamma(\cdot ; \bv)$ is strictly 
concave, we have $\bsalpha^* = \alphahatgv$.
By rearranging terms in the first-order necessary conditions, we conclude
${\bv - \gamma \alphahatgv \in \vhatA}$.

The contrapositive is: if $\bv - \gamma \alphahatgv \nin \vhatA$, then no 
minimizer exists.
Thus our strategy is to show that when $\bv$ has the form from 
\eqref{eq:junk-form-g}, we have ${\bv - \gamma \alphahatgv \nin \vhatA}$.

First, note that when $\bv$ has the form from \eqref{eq:junk-form-g}, then the
$\alphabar$ must be $\alphahatgv$.
To see this, recognize that $\alphabar \in \cA \cap \partial \cA$ implies 
that $\overline{\alpha}_{n - 1} = 0$, so that the concavity of $\psi_\gamma$ 
(and the requisite smoothness properties assured by \cite[Lemma 5.2]{Jun00}) 
gives
\begin{align}
 \psi_\gamma(\bsalpha; \bv) \le \psi_\gamma(\alphabar; \bv)
   + \psi'_\gamma(\alphabar; \bv) \cdot (\bsalpha - \alphabar)
  \stackrel{\eqref{eq:junk-form-g}}{=} \psi_\gamma(\alphabar; \bv)
   + \delta \alpha_{n - 1} 
  \le
   \psi_\gamma(\alphabar; \bv).
\end{align}
Since the maximizer of $\psi_\gamma(\cdot ; \bv)$ is unique,
$\alphabar = \alphahatgv$.
Thus \eqref{eq:junk-form-g} can be written as
\begin{align}
 \bv - \gamma \alphahatgv = \vhat(\alphahatgv)
  + \begin{pmatrix} 0 \\ \vdots \\ 0 \\ \delta \end{pmatrix},
\end{align}
and by \cite{Jun00}, the right-hand side is not in $\vhatA$.
\end{proof}

With a little more work, one can show that all degenerate densities
${\bv \in \R^n \setminus \vhatgA}$ have the form \eqref{eq:junk-form-g}.
Furthermore, for the case of more general polynomial basis functions, one 
can extend the arguments of \cite{Hauck-Levermore-Tits-2008} to show that the 
regularized problem satisfies the analogous complimentary-slackness condition
\begin{align}\label{eq:comp-slack}
 \alphahatgv \cdot (\bv - \vhatg(\alphahatgv)) = 0.
\end{align}
This complementary-slackness condition is the key to characterizing the set
degenerate densities (of the original problem) in
\cite{Hauck-Levermore-Tits-2008}.
Indeed, \eqref{eq:comp-slack} can be used to show that the set of all degenerate 
densities for the regularized problem is a union of normal cones which has the 
same form as \cite[Eq. (180)]{Hauck-Levermore-Tits-2008} with $\vhat$ (in that 
paper's notation, $\br$) replaced by $\vhatg$.

%% file: sections/ack.tex
\section*{Acknowledgements}

The authors would like to thank Prof.\ Benjamin Stamm for helpful discussions 
which led to the definition and analysis of the modified flux function 
$\bff_\gamma$.

Graham Alldredge's work was funded by the Deutsche Forschungsgemeinschaft, 
project ID AL 2030/1-1.
He would also like to warmly thank Prof.\ Ralf Kornhuber and his group at the 
Freie Universit\"at Berlin for kindly hosting his stay in Berlin.

This material is based, in part, upon work supported by the U.S. Department of 
Energy, Office of Science, Office of Advanced Scientific Computing and performed 
at Oak Ridge National Laboratory (ORNL), managed by UT-Battelle, LLC for the
U.S. Department of Energy under Contract No. De-AC05-00OR22725.

%% file: rmn-2017.bbl
\begin{thebibliography}{10}

\bibitem{AllHau12}
G.~Alldredge, C.~Hauck, and A.~Tits.
\newblock High-order entropy-based closures for linear transport in slab
  geometry {II}: A computational study of the optimization problem.
\newblock {\em SIAM Journal on Scientific Computing}, 34(4):B361--B391, 2012.

\bibitem{AlldredgeSchneider2014}
G.~Alldredge and F.~Schneider.
\newblock {A realizability-preserving discontinuous Galerkin scheme for
  entropy-based moment closures for linear kinetic equations in one space
  dimension}.
\newblock {\em Journal of Computational Physics}, 295:665--684, August 2015.

\bibitem{bianchini2002stability}
S.~Bianchini and R.~M. Colombo.
\newblock On the stability of the standard {R}iemann semigroup.
\newblock {\em Proceedings of the American Mathematical Society}, pages
  1961--1973, 2002.

\bibitem{Borwein-Lewis-1991}
J.~M. Borwein and A.~S. Lewis.
\newblock Duality relationships for entropy-like minimization problems.
\newblock {\em SIAM J. Control Optim.}, 1:191--205, 1991.

\bibitem{boyd2004convex}
S.~Boyd and L.~Vandenberghe.
\newblock {\em Convex optimization}.
\newblock Cambridge University Press, 2004.

\bibitem{BruHol01}
T.~A. Brunner and J.~P. Holloway.
\newblock One-dimensional {R}iemann solvers and the maximum entropy closure.
\newblock {\em J. Quant. Spectrosc. Radiat. Transfer}, 69:543--566, 2001.

\bibitem{caflisch1986equilibrium}
Russel~E. Caflisch and C.~David Levermore.
\newblock Equilibrium for radiation in a homogeneous plasma.
\newblock {\em The Physics of fluids}, 29(3):748--752, 1986.

\bibitem{Cercignani}
C.~Cercignani.
\newblock {\em The {B}oltzmann Equation and Its Applications}.
\newblock Springer-Verlag New York, New York, 1988.

\bibitem{CKS2000}
Bernardo Cockburn, George~E. Karniadakis, and Chi-Wang Shu.
\newblock {\em {Discontinuous Galerkin Methods: Theory, Computation and
  Applications}}.
\newblock Springer-Verlag, 2000.

\bibitem{CockburnShuIII}
Bernardo Cockburn, San-Yih Lin, and Chi-Wang Shu.
\newblock {TVB} {R}unge-{K}utta local projection discontinuous {G}alerkin
  finite element method for conservation laws {III}: {O}ne-dimensional systems.
\newblock {\em Journal of Computational Physics}, 84(1):90 -- 113, 1989.

\bibitem{Decarreau-Hilhorst-Lemarichal-Navaza-1992}
A.~Decarreau, D.~Hilhorst, C.~Lemar\'echal, and J.~Navaza.
\newblock Dual methods in entropy maximization. {A}pplication to some problems
  in crystallography.
\newblock {\em SIAM Journal on Optimization}, 2(2):173--197, 1992.

\bibitem{dimarco2013asymptotic}
Giacomo Dimarco and Lorenzo Pareschi.
\newblock Asymptotic preserving implicit-explicit {R}unge--{K}utta methods for
  nonlinear kinetic equations.
\newblock {\em SIAM Journal on Numerical Analysis}, 51(2):1064--1087, 2013.

\bibitem{Dubroca-Feugas-1999}
B.~Dubroca and J.-L. Fuegas.
\newblock {\'E}tude th\'eorique et num\'erique d'une hi\'erarchie de mod\`eles
  aus moments pour le transfert radiatif.
\newblock {\em C.R. Acad. Sci. Paris}, I. 329:915--920, 1999.

\bibitem{engl1989convergence}
H.~W. Engl, K.~Kunisch, and A.~Neubauer.
\newblock Convergence rates for {T}ikhonov regularisation of non-linear
  ill-posed problems.
\newblock {\em Inverse problems}, 5(4):523--540, 1989.

\bibitem{engl1993convergence}
H.~W. Engl and G.~Landl.
\newblock Convergence rates for maximum entropy regularization.
\newblock {\em SIAM Journal on Numerical Analysis}, 30(5):1509--1536, 1993.

\bibitem{evans2010partial}
L.~C. Evans.
\newblock {\em Partial Differential Equations}.
\newblock Graduate studies in mathematics. American Mathematical Society, 2010.

\bibitem{ganapol1977generation}
B.~D. Ganapol, P.~W. McKenty, and K.~L. Peddicord.
\newblock The generation of time-dependent neutron transport solutions in
  infinite media.
\newblock {\em Nuclear Science and Engineering}, 64(2):317--331, 1977.

\bibitem{Hauck-Levermore-Tits-2008}
Cory~D. Hauck, C.~David Levermore, and Andr\'{e}~L. Tits.
\newblock Convex duality and entropy-based moment closures: Characterizing
  degenerate densities.
\newblock {\em SIAM J. Control Optim.}, 47(4):1977--2015, 2008.

\bibitem{hu2017asymptotic}
Jingwei Hu, Ruiwen Shu, and Xiangxiong Zhang.
\newblock Asymptotic-preserving and positivity-preserving implicit-explicit
  schemes for the stiff {BGK} equation.
\newblock {\em arXiv preprint arXiv:1708.06279}, 2017.

\bibitem{jin1999ap}
Shi Jin.
\newblock Efficient asymptotic-preserving ({AP}) schemes for some multiscale
  kinetic equations.
\newblock {\em SIAM Journal on Scientific Computing}, 21(2):441--454, 1999.

\bibitem{Junk-1998}
M.~Junk.
\newblock Domain of definition of {L}evermore's five moment system.
\newblock {\em J. Stat. Phys.}, 93(5-6):1143--1167, 1998.

\bibitem{Jun00}
M.~Junk.
\newblock Maximum entropy for reduced moment problems.
\newblock {\em Math. Meth. Mod. Appl. Sci.}, 10:1001--1025, 2000.

\bibitem{JunUnt02}
M.~Junk and A.~Unterreiter.
\newblock Maximum entropy moment systems and {G}alilean invariance.
\newblock {\em Continuum Mech. Thermodyn.}, 14:563--576, 2002.

\bibitem{ketcheson2008highly}
David~I. Ketcheson.
\newblock Highly efficient strong stability-preserving {R}unge--{K}utta methods
  with low-storage implementations.
\newblock {\em SIAM Journal on Scientific Computing}, 30(4):2113--2136, 2008.

\bibitem{LasserreBook}
J.~B. Lasserre.
\newblock {\em {Moments, Positive Polynomials and Their Applications}}.
\newblock Imperial College Press, Singapore, 2010.

\bibitem{Lev96}
C.~D. Levermore.
\newblock Moment closure hierarchies for kinetic theories.
\newblock {\em J. Stat. Phys.}, 83:1021--1065, 1996.

\bibitem{Lewis-Miller-1984}
E.~E. Lewis and W.~F. Miller, Jr.
\newblock {\em Computational Methods in Neutron Transport}.
\newblock John Wiley and Sons, New York, 1984.

\bibitem{markowich1990}
Peter~A. Markowich, Christian~A. Ringhofer, and Christian Schmeiser.
\newblock {\em Semiconductor Equations}.
\newblock Springer-Verlag Wien, 1990.

\bibitem{mcclarren2008semi}
Ryan~G. McClarren, Thomas~M. Evans, Robert~B. Lowrie, and Jeffery~D. Densmore.
\newblock Semi-implicit time integration for ${P}_{N}$ thermal radiative
  transfer.
\newblock {\em Journal of Computational Physics}, 227(16):7561--7586, 2008.

\bibitem{mcclarren2010robust}
Ryan~G. McClarren and Cory~D. Hauck.
\newblock Robust and accurate filtered spherical harmonics expansions for
  radiative transfer.
\newblock {\em Journal of Computational Physics}, 229(16):5597--5614, 2010.

\bibitem{mcclarren2010simulating}
Ryan~G. McClarren and Cory~D. Hauck.
\newblock Simulating radiative transfer with filtered spherical harmonics.
\newblock {\em Physics Letters A}, 374(22):2290--2296, 2010.

\bibitem{mihalas1999foundations}
Dimitri Mihalas and Barbara Weibel-Mihalas.
\newblock {\em Foundations of Radiation Hydrodynamics}.
\newblock Courier Corporation, 1999.

\bibitem{Min78}
G.~N. Minerbo.
\newblock Maximum entropy {E}ddington factors.
\newblock {\em J. Quant. Spectrosc. Radiat. Transfer}, 20:541--545, 1978.

\bibitem{Olbrant2012}
E.~Olbrant, C.~D. Hauck, and M.~Frank.
\newblock {A realizability-preserving discontinuous Galerkin method for the M1
  model of radiative transfer}.
\newblock {\em Journal of Computational Physics}, 231(17):5612--5639, July
  2012.

\bibitem{SchneiderAlldredge2016}
Florian Schneider, Graham Alldredge, and Jochen Kall.
\newblock {A realizability-preserving high-order kinetic scheme using WENO
  reconstruction for entropy-based moment closures of linear kinetic equations
  in slab geometry.}
\newblock {\em Kinetic and Related Models}, 9(1), 2016.

\bibitem{schneider2004entropic}
Jacques Schneider.
\newblock Entropic approximation in kinetic theory.
\newblock {\em ESAIM: Mathematical Modelling and Numerical Analysis},
  38(3):541--561, 2004.

\bibitem{Shu1998}
Chi-Wang Shu.
\newblock Essentially non-oscillatory and weighted essentially non-oscillatory
  schemes for hyperbolic conservation laws.
\newblock In Alfio Quarteroni, editor, {\em Advanced Numerical Approximation of
  Nonlinear Hyperbolic Equations}, volume 1697 of {\em Lecture Notes in
  Mathematics}, pages 325--432. Springer Berlin Heidelberg, 1998.

\end{thebibliography}
